\documentclass[11pt]{amsart}
\usepackage{amsmath,amsthm,amsfonts,amssymb,amscd,euscript,enumerate}
\usepackage{color}
\usepackage{listings}
\usepackage{graphics}
\usepackage{tikz}
\usepackage{algorithmic}
\input{xy}
\xyoption{all}
\usepackage[bookmarks=true,colorlinks=true, linkcolor=blue, citecolor=cyan]{hyperref}

\numberwithin{equation}{section}

\newtheorem{theorem}{Theorem}[section]
\newtheorem{proposition}[theorem]{Proposition}
\newtheorem{conjecture}[theorem]{Conjecture}
\newtheorem{corollary}[theorem]{Corollary}
\newtheorem{lemma}[theorem]{Lemma}

\theoremstyle{definition}
\newtheorem{remark}[theorem]{Remark}

\newtheorem{definition}[theorem]{Definition}

\let\oldmarginpar\marginpar
\renewcommand\marginpar[1]{\-\oldmarginpar[\raggedleft\small\sf
#1]{\raggedright\small\sf #1}}

\newcommand{\myAA}{\mathcal{A}}

\newcommand{\ZZ}{\mathbb{Z}}

\newcommand{\RR}{\mathbb{R}}
\newcommand{\QQ}{\mathbb{Q}}

\newcommand{\sgn}{\operatorname{sgn}}

\newcommand{\uc}{\bar{e}}
\newcommand{\lc}{\underline{e}}
\newcommand{\ulc}{\bar{\underline{e}}}

\newcommand{\UX}{\overline{X}}
\newcommand{\LX}{\underline{X}}
\newcommand{\ULX}{\overline{\underline{X}}}

\newcommand{\cA}{\mathcal{A}}

\newcommand{\cT}{\mathcal{T}}

\begin{document}

\title
{The existence of greedy bases in rank 2 quantum cluster algebras}

\author{Kyungyong Lee}
\address{Department of Mathematics, Wayne State University, Detroit, MI 48202, U.S.A.}
\author{Li Li}
\address{Department of Mathematics and Statistics, Oakland University, Rochester, MI 48309, U.S.A.}
\author{Dylan Rupel}
\address{Department of Mathematics, Northeastern University, Boston, MA 02115, U.S.A.}
       \author{Andrei Zelevinsky}

\begin{abstract}
 We establish the existence of a quantum lift of the greedy basis.
\end{abstract}

\maketitle

\section{introduction}
At the heart of the definition of quantum cluster algebras is a desire to understand nice bases in quantum algebras arising from the representation theory of non-associative algebras.  Of particular interest is the dual canonical basis in the quantized coordinate ring of a unipotent group, or more generally in the quantized coordinate ring of a double Bruhat cell.  Through a meticulous study of these algebras and their bases the notion of a cluster algebra was discovered by Fomin and Zelevinsky \cite{fomin-zelevinsky1} with the notion of a quantum cluster algebra following in the work \cite{quantum} of Berenstein and Zelevinsky.  Underlying the definition of quantum cluster algebras is a deep conjecture that the quantized coordinate rings described above in fact have the structure of a quantum cluster algebra and that the cluster monomials arising from these cluster structures belong to the dual canonical bases of the quantum algebras.  The most pressing questions in the theory are thus related to understanding bases of a (quantum) cluster algebra.

Several bases are already known for both classical and quantum cluster algebras.  Our main interest in bases of classical cluster algebras is related to the positivity phenomenon observed in the dual canonical basis \cite{Lu}.  The constructions of interest originated with the atomic bases in finite types and affine type $A$ (in the sense of \cite{fomin-zelevinsky2}) which were discovered and constructed explicitly through the works of several authors.  These atomic bases consist of all indecomposable positive elements of the cluster algebra, they are ``atomic" in the sense that an indecomposable positive element cannot be written as a sum of two nonzero positive elements.

This line of investigation was initiated by Sherman and Zelevinsky in \cite{sz-Finite-Affine} where atomic bases were constructed in type $A_1^{(1)}$, the so called Kronecker type.  In the works \cite{cerulli-irelli1,cerulli-irelli2} atomic bases were constructed by Cerulli Irelli for finite types and type $A_2^{(1)}$ respectively.  The construction of atomic bases for affine type $A$ was completed by Dupont and Thomas in \cite{dupont-thomas} using triangulations of surfaces with marked points on the boundary.  The representation theory of quivers also played prominently in these works but since we will not pursue this direction here we refer the reader to the previously cited works for more details.

Pushing beyond affine types it was shown by Lee, Li, and Zelevinsky in \cite{LLZ2} that for wild types the set of all indecomposable positive elements can be linearly dependent and therefore do not form a basis.  In contrast, these authors in \cite{llz} constructed for rank 2 coefficient-free cluster algebras a  combinatorially defined ``greedy basis" which consists of a certain subset of the indecomposable positive elements.

Our main goal in this note is to establish the existence of a quantum lift of the greedy basis. 

\subsection{Rank 2 cluster algebras and their greedy bases}\label{sec:greedy}

Fix positive integers $b,c>0$.  The commutative cluster algebra $\cA(b,c)$ is the subring of $\QQ(x_1,x_2)$ generated by the \emph{cluster variables} $x_m$ ($m\in\ZZ$), where the $x_m$ are rational functions in $x_1$ and $x_2$, defined recursively by the \emph{exchange relations}
\[x_{m+1}x_{m-1}=\begin{cases} x_m^b+1 & \text{ if $m$ is odd;}\\ x_m^c+1 & \text{ if $m$ is even.}\end{cases}\]
The cluster variables are organized into clusters $\{x_m,x_{m+1}\}$ ($m\in\ZZ$).  

It is a fundamental result of Fomin and Zelevinsky \cite{fomin-zelevinsky1} that, although the exchange relations appear to produce rational functions, one always obtains Laurent polynomials.  They actually showed the following slightly stronger result which asserts in addition that the Laurent Phenomenon does not depend on the choice of an initial cluster.

\begin{theorem}[{\cite[Theorem 3.1]{fomin-zelevinsky1}, Laurent Phenomenon}]
 For any $m\in\ZZ$ we have $\cA(b,c)\subset\ZZ[x_m^{\pm1},x_{m+1}^{\pm1}]$.
\end{theorem}

The cluster algebra $\cA(b,c)$ is of \emph{finite type} if the collection of all cluster variables is a finite set.  In \cite{fomin-zelevinsky2} Fomin and Zelevinsky classified cluster algebras of finite type.

\begin{theorem}[{\cite[Theorem 1.4]{fomin-zelevinsky2}}]
 The cluster algebra $\cA(b,c)$ is of finite type if and only if $bc\le3$.
\end{theorem}

The proof of this theorem establishes a connection to the theory of rank 2 Kac-Moody root systems.  Thus we say $\cA(b,c)$ is of \emph{affine} (resp. \emph{wild}) type if $bc=4$ (resp. $bc\ge5$).  

An element $x\in\bigcap_{m\in\ZZ}\ZZ[x_m^{\pm1},x_{m+1}^{\pm1}]$ is called \emph{universally Laurent} since the expansion of $x$ in each cluster is a Laurent polynomial.  If the coefficients of the Laurent expansions of $x$ in each cluster are positive integers, then $x$ is called \emph{universally positive}.  A nonzero universally positive element in $\cA(b,c)$ is said to be \emph{indecomposable} if it cannot be expressed as a sum of two nonzero universally positive elements.

One of the main results of \cite{sz-Finite-Affine} is the following: if the commutative cluster algebra $\cA(b,c)$ is of finite or affine type, then the indecomposable universally positive elements form a $\ZZ$-basis in $\cA(b,c)$, moreover this basis contains all \emph{cluster monomials} $x_m^{a_1}x_{m+1}^{a_2}$ ($m\in\ZZ$, $a_1,a_2\in\ZZ_{\ge0}$).  However, in the wild types the situation becomes much more complicated; in particular, it is shown by Lee, Li, and Zelevinsky in \cite{LLZ2} that for $bc\ge5$ the indecomposable universally positive elements of $\cA(b,c)$ are linearly dependent.  The \emph{greedy basis} of $\cA(b,c)$ introduced in \cite{llz} is a subset  (a proper subset in wild types) of the indecomposable positive elements which has a manifestly positive combinatorial description. 

The elements of the greedy basis take on a particular form, which is motivated by a well-known pattern in the initial cluster expansion of cluster monomials.  An element $x \in \mathcal{A}(b,c)$ is \emph{pointed} at $(a_1, a_2) \in \ZZ^2$ if it can be written in the form
\begin{equation}
\label{eq:pointed-expansion}
x=x_1^{-a_1}  x_2^{-a_2} \sum_{p,q \geq 0} e(p,q) x_1^{bp} x_2^{cq}	
\end{equation}
with $e(0,0)=1$ and $e(p,q):=e_{a_1,a_2}(p,q) \in \ZZ$ for all $p,q\ge0$.

The following two theorems summarize the results from \cite{llz}.

\begin{theorem}
\label{proposition 1.6 of llzp}
For each $(a_{1},a_{2})\in\ZZ^{2}$, there exists a unique element in $\mathcal{A}(b,c)$ pointed at $(a_1,a_2)$ whose coefficients $e(p,q)$ satisfy the following recurrence relation:
\begin{equation}\label{eq:classical recurrence}
e(p,q)=  \begin{cases}
          \sum\limits_{k=1}^p(-1)^{k-1}e(p-k,q){ [a_2-cq]_++k-1\choose k} & \text{ if $ca_1q\le ba_2p$;}\\
          \sum\limits_{\ell=1}^q(-1)^{\ell-1}e(p,q-\ell){[a_1-bp]_++\ell-1\choose \ell} & \text{ if $ca_1q\ge ba_2p$.}
         \end{cases}
\end{equation}
 where we use the standard notation $[a]_+=\max(a,0)$.
\end{theorem} 

We define the {\it greedy element} pointed at $(a_1,a_2)$, denoted $x[a_1,a_2]$, as the unique element determined by Theorem 
\ref{proposition 1.6 of llzp}.

\begin{theorem}
\label{main theorem-commutative}
Fix positive integers $b,c$.


{\rm(a)} The greedy elements $x[a_1,a_2]$ for $(a_1, a_2) \in \ZZ^2$ form a $\ZZ$-basis in $\cA(b,c)$, which we refer to as the \emph{greedy basis}.

{\rm(b)} The greedy basis is independent of the choice of an initial cluster.

{\rm(c)} The greedy basis contains all cluster monomials.

{\rm(d)}  Greedy elements are universally positive and indecomposable.
\end{theorem}

Our goal in this work is to generalize the above two theorems to the setting of rank 2 quantum cluster algebras.  The proof of Theorem~\ref{main theorem-commutative} given in \cite{llz} uses combinatorial objects called \emph{compatible pairs} in an essential way (cf. \cite[Theorem 1.11]{llz}). Unfortunately this method has limitations in generalizing to the study of quantum cluster algebras. More precisely, if one could construct quantum greedy elements by assigning powers of $q$ to each compatible pair then the quantum greedy elements would have to be universally positive, which is unfortunately false (for instance, see  \cite[Section 3]{llrz}). So we shall take a completely different approach to prove the quantum analogue of this theorem.

\subsection{Rank 2 quantum cluster algebras}

We now define our main objects of study, namely quantum cluster algebras, and recall important fundamental facts related to these algebras.  We restrict our attention to rank 2 quantum cluster algebras where we can describe the setup in very concrete terms.  We follow (as much as possible) the notation and conventions of \cite{llz,triangular}.

 We  work in the quantum torus
 $$\cT:=\ZZ[v^{\pm 1}]\langle X_1^{\pm1},X_2^{\pm1}: X_2 X_1=v^2 X_1 X_2\rangle$$
 (this setup is related to the one in \cite{rupel1} which uses the formal variable $q$ instead of $v$ by setting $q = v^{-2}$).  There are many choices for quantizing cluster algebras, to rigidify the situation we require the quantum cluster algebra to be invariant under a certain involution.
The \emph{bar-involution} is the $\mathbb{Z}$-linear anti-automorphism of $\cT$ determined by $\overline{f}(v)=f(v^{-1})$ for  $f\in\mathbb{Z}[v^{\pm1}]$ and
\[\overline{fX_1^{a_1}X_2^{a_2}}=\overline{f}X_2^{a_2}X_1^{a_1}=v^{2a_1a_2}\overline{f}X_1^{a_1}X_2^{a_2}\quad\text{($a_1,a_2\in\ZZ$)}.\]
An element which is invariant under the bar-involution is said to be \emph{bar-invariant}.

Let $\mathcal{F}$ be the skew-field of fractions
of $\cT$.    The \emph{quantum cluster algebra} $\cA_v(b,c)$ is the $\mathbb{Z}[v^{\pm1}]$-subalgebra of $\mathcal{F}$ generated by the \emph{cluster variables} $X_m$ ($m \in\ZZ$) defined recursively by the \emph{exchange relations}
 \[X_{m+1}X_{m-1}=\begin{cases} v^{b}X_m^b+1 & \text{ if $m$ is odd;}\\ v^{c}X_m^c+1 & \text{ if $m$ is even.}\end{cases}\]

 By a simple induction one can easily check the following (quasi-)commutation relations between neighboring cluster variables $X_m,X_{m+1}$ in $\cA_v(b,c)$:
 \begin{equation}\label{eq:commutation-A11}
  X_{m+1} X_m = v^2 X_m X_{m+1} \quad (m \in \ZZ).
 \end{equation}
 It then follows from an easy induction that all cluster variables are bar-invariant, therefore $\cA_v(b,c)$ is also stable under the bar-involution.  Equation \eqref{eq:commutation-A11} implies that each \emph{cluster} $\{X_m, X_{m+1}\}$ generates a quantum torus
\[\cT_m:=\ZZ[v^{\pm1}]\langle X_m^{\pm1},X_{m+1}^{\pm1}: X_{m+1} X_m = v^2 X_m X_{m+1}\rangle.\]
 It is easy to see that the bar-involution does not depend on the choice of an initial quantum torus $\cT_m$.

 The appropriate quantum analogues of cluster monomials for $\cA_v(b,c)$ are the (bar-invariant) \emph{quantum cluster monomials} which are normalized monomials in the quantum tori $\cT_m$, more precisely they are
 \[X^{(a_1,a_2)}_m = v^{a_1 a_2} X_m^{a_1} X_{m+1}^{a_2} \quad (a_1, a_2 \in \ZZ_{\geq 0}, \,\, m \in \ZZ).\]
 For convenience we will often abbreviate $X^{(a_1,a_2)}:=X^{(a_1,a_2)}_1$. 

The following quantum analogue of the Strong Laurent Phenomenon was proven by Berenstein and Zelevinsky in \cite{quantum}.

 \begin{theorem}[{\cite[Theorem 5.1, Corollary 5.2 and Theorem 7.5]{quantum}}]\label{thm1}
  For any $m\in\ZZ$ we have $\cA_v(b,c)\subset\cT_m$.  Moreover, $$\cA_v(b,c)=\bigcap\limits_{m\in\ZZ}\cT_m=\bigcap\limits_{m=0}^{2}\cT_m.$$
 \end{theorem}

A nonzero element of $\cA_v(b,c)$ is called {\it universally positive} if it lies  in $$\bigcap\nolimits_{m\in\ZZ}\ZZ_{\geq 0}[v^{\pm1}][X_m^{\pm1},X_{m+1}^{\pm1}].$$ 
A universally positive element in $\cA_v(b,c)$ is {\it indecomposable} if it cannot be expressed as a sum of two universally positive elements.

 As a very special case of the results of \cite{Q, Ef, DMSS, N-quiver, KQ} one may see that  cluster monomials are universally positive Laurent polynomials.
Explicit combinatorial expressions for the positive coefficients can be obtained from the results of \cite{ls, rupel2, LL:qgrass}.


\subsection{Quantum greedy bases}

This section contains the main result of the paper. Here we introduce the quantum greedy basis and present its nice properties.

Similar to the greedy elements, the elements of the quantum greedy basis take on the following particular form.
An element $X \in \mathcal{T}$ (resp. $X\in\mathcal{T}\otimes_\ZZ \mathbb{Q}$) is said to be {\it pointed} at $(a_1, a_2) \in \ZZ^2$ if it has the form
\begin{equation}
\label{eq:pointed-expansion-q}
X=\sum\limits_{p,q\ge0} e(p,q)X^{(-a_1+bp,-a_2+cq)}
\end{equation}
with $e(0,0)=1$ and $e(p,q)\in\ZZ[v^{\pm1}]$ (resp. $e(p,q)\in\QQ[v^{\pm1}]$) for all $p$ and $q$.
All quantum cluster variables, hence all quantum cluster monomials, are pointed.

For $n,k\in\ZZ$ and for $w=v^{b}$ or $v^{c}$, define the bar-invariant quantum numbers and quantum binomial coefficients by
 \[
 \aligned
 {}[n]_w&=\frac{w^n-w^{-n}}{w-w^{-1}}=\sgn(n)\big(w^{|n|-1}+w^{|n|-3}+\cdots+w^{-|n|+1});\\
 {n\brack k}_w&=\frac{[n]_w[n-1]_w\cdots[n-k+1]_w}{[k]_w[k-1]_w\cdots[1]_w}.
 \endaligned
 \]
 Note that ${n\brack k}_w$ is a Laurent polynomial in $w$ and hence in $v$ as well.  It is immediate that $[-n]_w=-[n]_w$ and thus ${-n\brack k}_w=(-1)^k{n+k-1\brack k}_w$.

\begin{definition}\label{le,ge}  A pointed element $\sum e(p,q)X^{(-a_1+bp,-a_2+cq)}$ is said to satisfy the recurrence relation $(\le,\ge)$  if the following recursion is satisfied for $(p,q)\ne(0,0)$:
  \begin{equation}\label{eq:recurrence}
e(p,q)=  \begin{cases}
          \sum\limits_{k=1}^p(-1)^{k-1}e(p-k,q){[a_2-cq]_++k-1\brack k}_{v^b} & \text{ if }ca_1q\le ba_2p;\\
          \sum\limits_{\ell=1}^q(-1)^{\ell-1}e(p,q-\ell){[a_1-bp]_++\ell-1\brack \ell}_{v^c} & \text{ if }ca_1q\ge ba_2p.\\
         \end{cases}
\end{equation}

We denote by $(\le,>)$  the recurrence relation obtained from  \eqref{eq:recurrence} by replacing ``$\ge$'' by ``$>$'',  and $(<,\ge)$ the recurrence relation obtained from  \eqref{eq:recurrence} by replacing ``$\le$'' by ``$<$''. 
\end{definition}

\begin{theorem}\label{main theorem1}
For each $(a_{1},a_{2})\in\ZZ^{2}$, there exists a unique element pointed at $(a_1,a_2)$ whose coefficients $e(p,q)$ satisfy the recurrence relation $(\le,\ge)$. 
\end{theorem}
\noindent Such an element is called a {\it quantum greedy element} and will be denoted by $X[a_{1},a_{2}]$.
\begin{corollary}\label{cor:greedy in cluster algebra}
  Each quantum greedy element $X[a_{1},a_{2}]$ is in $\mathcal{A}_v(b,c)$.
\end{corollary}

The following theorem asserts that the quantum greedy elements possess all the desired properties described in Theorem \ref{main theorem-commutative} except for the positivity in part (d).

\begin{theorem}[Main Theorem]\label{main theorem2}
Fix positive integers $b,c$.

{\rm(a)} The quantum greedy elements $X[a_1,a_2]$ for $(a_1, a_2) \in \ZZ^2$ form a $\ZZ[v^{\pm1}]$-basis in $\cA_v(b,c)$, which we refer to as the \emph{quantum greedy basis}.

{\rm(b)} The quantum greedy basis is bar-invariant and independent of the choice of an initial cluster.

{\rm(c)} The quantum greedy basis contains all cluster monomials.

{\rm(d)} If $X[a_1,a_2]$ is universally positive, then it is indecomposable.

{\rm(e)} The quantum greedy basis specializes to the commutative greedy basis by the substitution $v=1$.
\end{theorem}

The proof of Theorem \ref{main theorem1} and \ref{main theorem2} is given at the end of \S5. Regarding Theorem \ref{main theorem2}(d), we propose the following conjecture based on extensive computation.

\begin{conjecture}
If $b|c$ or $c|b$, then all quantum greedy elements are universally positive, hence indecomposable.
\end{conjecture}

The structure of the paper is as follows. Section 2 studies certain $\mathbb{Z}$-linear automorphisms on $\mathcal{A}_v(b,c)$.  Section 3 gives a non-recursive characterization of quantum greedy elements which highlights the importance of the observations of Section 2. Section 4 defines the weaker version of quantum greedy elements, namely the upper (and lower) quasi-greedy elements, and gives a non-recursive characterization for them similar to that given in Section 2.  We also show in Section 4 that an upper or lower quasi-greedy element is sent to an upper or lower quasi-greedy element by the automorphisms defined in Section 2. Section 5 studies the quasi-greedy elements, in particular we show that the upper and lower quasi-greedy elements are indeed identical and consequently prove the main theorems (Theorems \ref{main theorem1} and \ref{main theorem2}). In the Appendix (Section 6) we remind the reader of the quantum binomial theorem.\\

{\bf Acknowledgments:}
Most of the ideas toward this work from D.~R. and A.~Z. were had during their stay at the Mathematical Sciences Research Institute as part of the Cluster Algebras Program.  They would like to thank the MSRI for their hospitality and support.
Research supported in part by NSF grants DMS-0901367 (K.~L.) and DMS-1103813 (A.~Z.),  and by Oakland University URC Faculty Research Fellowship Award (L.~L.).
The authors would like to thank F.~Qin for valuable discussions. 

Sadly Andrei Zelevinsky passed away in the early stages of writing.  We hope to have achieved the clarity of exposition that Andrei's artful eye would have provided.

\section{Automorphisms acting on rank 2 quantum cluster algebras}\label{sec:symmetries}

We consider in this section a quantum analogue of the group of automorphisms of $\myAA(b,c)$  introduced in \cite{llz}. 

\begin{lemma}\label{automorphism group:q}
For each integer $\ell$, there is a well-defined $\mathbb{Z}$-linear involutive automorphism $\sigma_\ell$ on $\cT$ satisfying
$$\sigma_\ell(X_m) = X_{2 \ell-m},\quad \sigma_\ell(f(v))=f(v^{-1}).$$
It specializes to a $\mathbb{Z}$-linear automorphism of $\myAA(b,c)$. The group  $W$ generated by $\{\sigma_\ell\}_{\ell\in\mathbb{Z}}$ is a dihedral group generated by $\sigma_1$ and $\sigma_2$, and is finite if and only if $bc\le 3$.
\end{lemma}
\begin{proof}
It suffices to check that $\sigma_\ell(X_{m+1}X_m)=v^{-2}\sigma_\ell(X_mX_{m+1})$. The proof is straightforward.
\end{proof}

Let $Y=\sum\limits_{p,q\ge0} d(p,q)X^{(-a_1+bp,-a_2+cq)}$ be an element in $\cT$ or $\cT\otimes_\ZZ\QQ$ pointed at $(a_1,a_2)$. We say that $Y$ satisfies the {\it divisibility condition} if
\begin{equation}\label{eq:divisibility}
\;  \left\{\begin{array}{ll} \sum\limits_{k\ge0}{a_2-cq\brack k}_{v^c}t^k\ \bigg|\ \sum\limits_{p\ge0} d(p,q)t^p,\text{\; for every }0\le q< a_2/c;\\
\sum\limits_{\ell\ge0}{a_1-bp\brack \ell}_{v^c}t^\ell \;\;\!\! \ \bigg| \ \sum\limits_{q\ge0} d(p,q)t^q,\text{\;  for every }0\le p<a_1/b. \end{array}\right.
\end{equation}

\begin{lemma}\label{lem:divisibility}
Let $Y$ be as above. 

{\rm(1)}  $\sigma_1(Y)\in\cT$  if and only if the first condition of \eqref{eq:divisibility} holds.

{\rm(2)} $\sigma_2(Y)\in\cT$  if and only if the second condition of \eqref{eq:divisibility} holds.

As a consequence, $Y\in A_v(b,c)$ if and only if the divisibility condition \eqref{eq:divisibility}  holds.  
\end{lemma}
\begin{proof}
We only prove (2), the proof of (1) is similar.  For every $0\le p\le a_2$, set $p'=a_2-p$, $a'_1=-a_1+ba_2$, and define $\{d'(p',q)\}_{q\ge0}$ using the power series expansion
\begin{equation}\label{eq:d'}
\sum_{q\ge0} d'(p',q)t^q=\Bigg(\sum\limits_{\ell\ge0}{-a_1+bp\brack \ell}_{v^c}t^\ell\Bigg)\Bigg(\sum_{q\ge0} d(p,q)t^q\Bigg),
\end{equation}

We claim that the following identity holds for every $p\ge0$:
\begin{equation}\label{eq:sigma2:q}
\sigma_2\Big(\sum_{q\ge0} d(p,q)X^{(-a_1+bp,-a_2+cq)}\Big)=
\sum_{q\ge0} d'(p',q)X^{(-a'_1+bp',-a_2+cq)}.
\end{equation}
Indeed, since $X_3=(v^cX_2^c+1)X_1^{-1}$, by Lemma~\ref{le:quantum binomial theorem} we have
\begin{align}\label{eq:variable power}
 X_3^n&=\sum\limits_{\ell\ge0}{n\brack \ell}_{v^c} v^{-nc\ell}X_1^{-n}X_2^{c\ell}
\end{align}
so that the left hand side of \eqref{eq:sigma2:q} is
$$\aligned
&=\sum_{q\ge0} d(p,q)v^{-(-a_1+bp)(-a_2+cq)}X_3^{-a_1+bp}X_2^{-a_2+cq}\\
&=\Bigg(\sum\limits_{\ell\ge0}{-a_1+bp\brack \ell}_{v^c}v^{(a_1-bp)c\ell}X_1^{a_1-bp}X_2^{c\ell}\Bigg)\Bigg(\sum_{q\ge0} d(p,q)v^{(a_1-bp)(-a_2+cq)}X_2^{-a_2+cq}\Bigg)\\
&=X_1^{a_1-bp}(v^{a_1-bp}X_2)^{-a_2}\Bigg(\sum\limits_{\ell\ge0}{-a_1+bp\brack \ell}_{v^c}\big(v^{(a_1-bp)c}X_2^c\big)^\ell\Bigg)\Bigg(\sum_{q\ge0} d(p,q)\big(v^{(a_1-bp)c}X_2^c\big)^q\Bigg)\\
&\stackrel{\eqref{eq:d'}}{=}X_1^{a_1-bp}(v^{a_1-bp}X_2)^{-a_2}\sum_{q\ge0} d'(p',q)\big(v^{(a_1-bp)c}X_2^c\big)^q\\
&=\sum_{q\ge0} d'(p',q)v^{(a_1-bp)(-a_2+cq)}X_1^{a_1-bp}X_2^{-a_2+cq},\\
\endaligned
$$
which is equal to the right hand side of \eqref{eq:sigma2:q}.

Combining the equalities \eqref{eq:sigma2:q} for all $p$ we get  
$$\sigma_2(Y)=\sum_{p',q\ge0}d'(p',q)X^{(-a'_1+bp',-a_2+cq)}.$$
It follows that $\sigma_2(Y)$ is in $\cT$ if and only if all but finitely many elements in $\{d'(p',q)\}_{p',q\ge0}$ are 0, in particular the left hand side of \eqref{eq:d'} must be finite.  Note that this finiteness is trivial for $bp\ge a_1$.  By Lemma~\ref{cor:quantum convolution} we have
\[\sum\limits_{\ell\ge0}{-a_1+bp\brack \ell}_{v^c}t^\ell=\Bigg(\sum\limits_{\ell\ge0}{a_1-bp\brack \ell}_{v^c}t^\ell\Bigg)^{-1}\]
and so finiteness is then equivalent to the divisibility condition in (2).  

The final claim follows immediately from Theorem \ref{thm1}.
\end{proof}

\section{An equivalent statement to Theorem~\ref{main theorem1}}

By analogy with the greedy elements $x[a_1,a_2]$ we will use the following definition in describing the support of our quantum greedy elements $X[a_1,a_2]$.

\begin{definition}\cite{llz}\label{df:PSR}
 For integers $a_1,a_2$, the \emph{pointed support region} $R_{\text{greedy}}[a_1,a_2]$ of $x[a_1,a_2]$ is a subset of $\mathbb{R}_{\ge0}^2$ defined in the following six cases.
\begin{enumerate}
\item If $a_1 \leq 0$ and $a_2 \leq 0$, then $R_{\text{greedy}}[a_1,a_2] := \{(0,0)\}$.

\item If $a_1 \leq 0 < a_2$, then $R_{\text{greedy}}[a_1,a_2] := \{(p,0)\in\mathbb{R}_{\ge0}^2\big|\; 0 \leq p \leq a_2\}$.

\item If $a_2 \leq 0 < a_1$, then $R_{\text{greedy}}[a_1,a_2] := \{(0,q)\in\mathbb{R}_{\ge0}^2\big|\; 0 \leq q \leq a_1\}$.

\item If $0<ba_2\leq a_1$, then
$$
R_{\text{greedy}}[a_{1},a_{2}]:=\big\{(p,q)\in\mathbb{R}_{\ge0}^2\big|\; 
p\le a_{2},\; q\le a_{1}-bp\big\}.
$$

\item If $0<ca_1\leq a_2$, then 
$$
R_{\text{greedy}}[a_{1},a_{2}]:=\big\{(p,q)\in\mathbb{R}_{\ge0}^2\big|\; 
q\le a_{1},\; p\le a_{2}-cq\big\}.
$$

\item If $0 < a_1 < ba_2$ and $0 < a_2 < ca_1$, then
$$\aligned
&R_{\text{greedy}}[a_{1},a_{2}]:=
\bigg\{(p,q)\in\mathbb{R}_{\ge0}^2\bigg|\; 0\le p< \frac{a_1}{b},\; q+\Big(b-\frac{ba_2}{ca_1}\Big)p<a_1\bigg\}\\
&\hspace{1.06in}\bigcup\bigg\{(p,q)\in\mathbb{R}_{\ge0}^2\bigg|\; 0\le q< \frac{a_2}{c},\; p+\Big(c-\frac{ca_1}{ba_2}\Big)q<a_2\bigg\}
\bigcup\Big\{(0,a_1),(a_2,0)\Big\}.
\endaligned
$$
\end{enumerate}
\end{definition}

\begin{remark}
(1) Here is a more geometric description (see Figure \ref{fig:PSR}). Denote
$O=(0,0), A=(a_2,0), B=(a_1/b, a_2/c), C=(0,a_1), D_{1}=(a_{2},a_{1}-ba_{2}), D_{2}=(a_{2}-ca_{1},a_{1}).$ Denote by $OAD_{1}C$ and $OAD_{2}C$ the regions of closed trapezoids, and denote by $OABC$ the region that includes the closed segments $OA$ and $OC$ but excludes the rest of the boundary. Then
$R_{\text{greedy}}[a_{1},a_{2}]$ is: $O$ in case (1); $OA$ in case (2); $OC$ in case (3); $OAD_{1}C$ in case (4); $OAD_{2}C$ in case (5); $OABC$ in case (6). Also note that $D_{1}\in OABC$ in case (4), and $D_{2}\in OABC$ in case (5).

(2) We compare some similar definitions we use in this paper and \cite{llz}.  For an element $x \in \mathcal{A}(b,c)$ pointed at $(a_1, a_2)$, express $x$ in two ways:
$$
x = \sum_{p,q}d(p,q)x_1^px_2^q\; =\; x_1^{-a_1}  x_2^{-a_2} \sum_{p,q \geq 0} e(p,q) x_1^{bp} x_2^{cq}
$$
The set $\{(p,q)\in\ZZ^{2}\; |\; d(p,q)\neq0 \}$ (resp. $\{(p,q)\in\ZZ_{\ge0}^{2}\; |\; e(p,q)\neq0 \}$)
is called the \emph{support} (resp. the \emph{pointed support}) of $x$. The support of $x$ is the image of the pointed support of $x$ under the transition map
$$\varphi: \mathbb{R}^2\to\mathbb{R}^2,\quad (p,q)\mapsto(-a_1+bp,-a_2+cq).$$
It was shown in \cite{llz} that the pointed support of the greedy element $x[a_1,a_2]$ is contained in the pointed support region $R_{\text{greedy}}[a_1,a_2]$, or equivalently, that the support of $x[a_1,a_2]$ is contained in the \emph{support region} of $x[a_1,a_2]$ defined as
$$S[a_1,a_2]:=\varphi(R_{\text{greedy}}[a_1,a_2])\subseteq \mathbb{R}^{2}.$$
\end{remark}

\begin{figure}[h]
\begin{tikzpicture}[scale=.7]
\begin{scope}[shift={(.5, 5.5)}]
\usetikzlibrary{patterns}
\draw (0,0) node[anchor=east] {\tiny $O$};
\draw[->] (0,0) -- (2.5,0)
node[above] {\tiny $p$};
\draw[->] (0,0) -- (0,2.5)
node[left] {\tiny $q$};
\fill (0,0) circle (2pt);
\draw (-1,-.5) circle (1.5pt);
\draw (-1,-.5) node[anchor=south] {\tiny $B$};
\draw (0.5,-.6) node[anchor=north] {\footnotesize (1) $a_1,a_2\le0$};
\end{scope}
\begin{scope}[shift={(6, 5.5)}]
\usetikzlibrary{patterns}
\draw (0,0) node[anchor=east] {\tiny$O$};
\draw (2,0) node[anchor=south] {\tiny$A$};
\draw[->] (0,0) -- (2.5,0)
node[above] {\tiny $p$};
\draw[->] (0,0) -- (0,2.5)
node[left] {\tiny $q$};
\fill (0,0) circle (1.5pt);
\fill (2,0) circle (1.5pt);
 \draw [very thick] (0,0) -- (2,0);
\draw (-1,.5) circle (1.5pt);
\draw (-1,.5) node[anchor=south] {\tiny $B$};
\draw (1,-.6) node[anchor=north] {\footnotesize (2) $a_1\le 0<a_2$};
\end{scope}
\begin{scope}[shift={(11.5,5.5)}]
\usetikzlibrary{patterns}
\draw (0,0) node[anchor=east] {\tiny$O$};
\draw (0,2) node[anchor=east] {\tiny$C$};
\draw[->] (0,0) -- (2.5,0)
node[above] {\tiny $p$};
\draw[->] (0,0) -- (0,2.5)
node[left] {\tiny $q$};
\fill (0,0) circle (1.5pt);
\fill (0,2) circle (1.5pt);
\draw [very thick] (0,0) -- (0,2);
\draw (1,-.3) circle (1.5pt);
\draw (1,-.3) node[anchor=west] {\tiny $B$};
\draw (1,-.6) node[anchor=north] {\footnotesize (3) $a_2\le0<a_1$};
\end{scope}
\usetikzlibrary{patterns}
\draw (0,3)--(1.5,1.5)--(1.5,0);
\draw[dashed](0,3)--(2.15,1.25) (2.15,1.15)--(1.5,0);
\draw (2.2,1.2) circle (2pt);
\draw (2.3,1.2) node[anchor=west] {\tiny$B$};
\fill [black!10] (0,3)--(1.5,1.5)--(1.5,0)--(0,0)--(0,3);
\draw (0,0) node[anchor=east] {\tiny$O$};
\draw (1.5,-.1) node[anchor=south west] {\tiny$A$};
\draw (1.4,1.5) node[anchor= south west] {\tiny$D_{1}$};
\draw (0,3) node[anchor=east] {\tiny$C$};
\draw[->] (0,0) -- (3.5,0)
node[above] {\tiny $p$};
\draw[->] (0,0) -- (0,3.5)
node[left] {\tiny $q$};
\fill (0,0) circle (1.5pt);
\fill (0,3) circle (1.5pt);
\fill (1.5,0) circle (1.5pt);
\fill (1.5,1.5) circle (1.5pt);
\draw (1.5,-.5) node[anchor=north] {\footnotesize (4) $0<ba_2\leq a_1$};
\begin{scope}[shift={(5.5,0)}]
\usetikzlibrary{patterns}
\draw (3,0)--(1.3,1)--(0,1);
\draw[dashed](3,0)--(0.95,1.45) (0.85,1.47)--(0,1);
\draw (0.9,1.5) circle (2pt);
\draw (0.9,1.5) node[anchor=south] {\tiny$B$};
\fill [black!10] (3,0)--(1.3,1)--(0,1)--(0,0)--(3,0);
\draw (0,0) node[anchor=east] {\tiny$O$};
\draw (2.7,0) node[anchor=south west] {\tiny$A$};
\draw (1.3,.9) node[anchor=south west] {\tiny$D_{2}$};
\draw (0,1) node[anchor=east] {\tiny$C$};
\draw[->] (0,0) -- (3.5,0)
node[above] {\tiny $p$};
\draw[->] (0,0) -- (0,3.5)
node[left] {\tiny $q$};
\fill (0,0) circle (1.5pt);
\fill (0,1) circle (1.5pt);
\fill (3,0) circle (1.5pt);
\fill (1.3,1) circle (1.5pt);
\draw (1.5,-.5) node[anchor=north] {\footnotesize (5) $0<ca_1\leq a_2$};
\end{scope}
\begin{scope}[shift={(11,0)}]
\usetikzlibrary{patterns}
\draw[dashed] (0,3)--(1.5,1.8)--(2.5,0);
\fill [black!10]  (0,3)--(1.5,1.77)--(2.5,0)--(0,0)--(0,3);
\draw (0,0) node[anchor=east] {\tiny$O$};
\draw (2.2,0) node[anchor=south west] {\tiny$A$};
\draw (1.5,1.8) node[anchor=south west] {\tiny$B$};
\draw (0,3) node[anchor=east] {\tiny$C$};
\draw[->] (0,0) -- (3.5,0)
node[above] {\tiny $p$};
\draw[->] (0,0) -- (0,3.5)
node[left] {\tiny $q$};
\fill (0,0) circle (1.5pt);
\fill (0,3) circle (1.5pt);
\fill (2.5,0) circle (1.5pt);
\draw (1.5,-.5) node[anchor=north] {\footnotesize (6) $0<a_1<ba_2$,};
\draw (1.5,-1) node[anchor=north] {\footnotesize \hspace{10pt} $0<a_2<ca_1$,};
\draw (1.5,-1.5) node[anchor=north] {\footnotesize \hspace{10pt} \tiny{$(a_1,a_2):$ non-imaginary root}};
\end{scope}
\begin{scope}[shift={(16,0)}]
\usetikzlibrary{patterns}
\draw[dashed] (0,3)--(.7,.78) (.78,.7)--(2.5,0);
\fill [black!10]  (0,3)--(.7,.7)--(2.5,0)--(0,0)--(0,3);
\draw (0,0) node[anchor=east] {\tiny$O$};
\draw (2.2,0) node[anchor=south west] {\tiny$A$};
\draw (.6,.6) node[anchor=south west] {\tiny$B$};
\draw (0,3) node[anchor=east] {\tiny$C$};
\draw[->] (0,0) -- (3.5,0)
node[above] {\tiny $p$};
\draw[->] (0,0) -- (0,3.5)
node[left] {\tiny $q$};
\fill (0,0) circle (1.5pt);
\fill (0,3) circle (1.5pt);
\fill (2.5,0) circle (1.5pt);
\draw (1.5,-.5) node[anchor=north] {\footnotesize (6) $0<a_1<ba_2$,};
\draw (1.5,-1) node[anchor=north] {\footnotesize \hspace{10pt} $0<a_2<ca_1$,};
\draw (1.5,-1.5) node[anchor=north] {\footnotesize \hspace{10pt} \tiny{imaginary root}};
\end{scope}
\end{tikzpicture}
\caption{\thickmuskip=0mu\small $R_{\text{greedy}}[a_1,a_2]$.} 
\label{fig:PSR}
\end{figure}
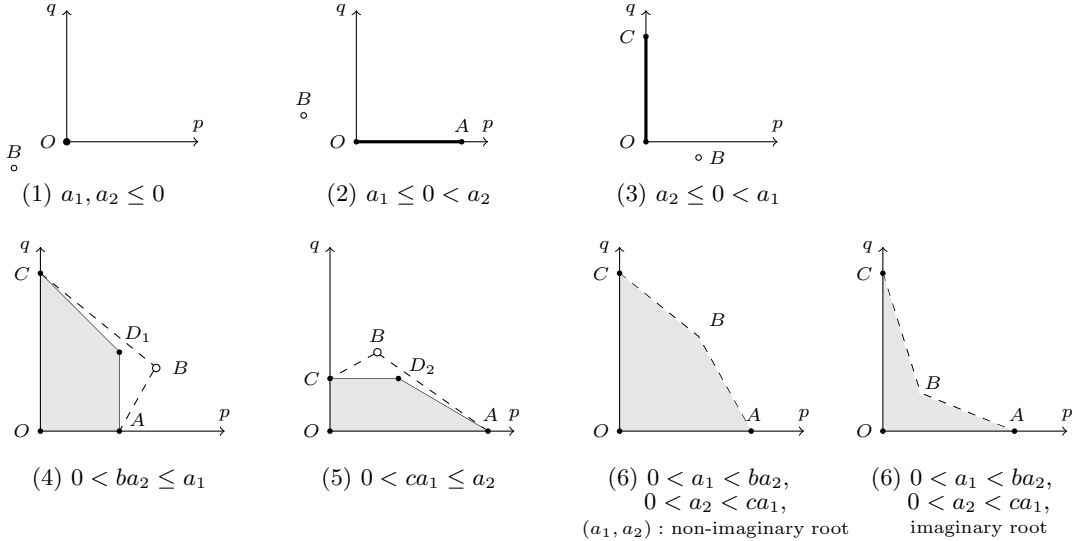

We aim to show that the quantum greedy elements defined by \eqref{eq:recurrence} satisfy the same support conditions.  The following general result will be useful for many of the cases.  In the following proposition and throughout the paper, the value of a vacuous product is by convention always equal to $1$. 
\begin{proposition}\label{gen_recursion}
Let $m,n$ be nonnegative integers,  $w=v^{d}\in\ZZ[v^{\pm1}]$ for a fixed integer $d$, and
$\{e_p\}_{p\ge0}$ be a sequence of elements in $\ZZ[v^{\pm1}]$ satisfying
\begin{equation}\label{general_recursion}
e_p=\sum\limits_{k=1}^p(-1)^{k-1}e_{p-k}{m+k-1\brack k}_w
\end{equation}
for all $p>n$. Then the sequence is uniquely determined by the first $n+1$ terms $e_{0},\dots,e_{n}$ and the following two conditions:
\smallskip

{\rm(a)} $e_p=0$ for all $p>n+m$,
\smallskip

{\rm(b)} $\sum\limits_{k=0}^m{m\brack k}_w t^k\ \bigg|\sum\limits_{p=0}^{m+n}e_pt^p.$
\end{proposition}
\begin{proof}
For $m=0$ there is nothing to show, so we assume $m\ge1$ for the rest of the proof.
Rewrite \eqref{general_recursion} as
\begin{equation}\label{new_eq:alternating sum=0}
\sum_{k=0}^p(-1)^k
e_{p-k}{m+k-1\brack k}_w=0.
\end{equation}
Consider two power series in $\ZZ[v^{\pm1}][[t]]$:
$$F(t):=\sum_{k\ge0}(-1)^k{m+k-1\brack k}_w t^k=\Bigg(\sum\limits_{k=0}^m{m\brack k}_w t^k\Bigg)^{-1}, \quad E(t):=\sum_{p\ge0} e_p t^p,$$
where the second equality defining $F(t)$ follows from Lemma~\ref{cor:quantum convolution}.  The left hand side of \eqref{new_eq:alternating sum=0} is the coefficient of $t^p$ in $F(t)E(t)$.
Thus \eqref{new_eq:alternating sum=0} being true for all $p>n$ is equivalent to saying that $F(t)E(t)$ is a polynomial in $t$ of degree at most $n$, which implies that 
$\Big(\sum\limits_{k=0}^m{m\brack k}_w t^k\Big)F(t)E(t)=E(t)$  is a polynomial in $t$ of degree at most $n+m$. Therefore $e_{p}=0$ for all $p>n+m$ and 
$\sum\limits_{k=0}^m{m\brack k}_w t^k$ divides $E(t)$. 

Next we show the uniqueness. If there is another sequence $\{e'_p\}_{p=0}^\infty$ with $e'_p=e_p$ for $0\le p\le n$ which also satisfies (a) and (b), then $\sum\limits_{k=0}^m{m\brack k}_w t^k$ divides
$$\sum\limits_{p\ge0}(e_p-e'_p)t^p=t^{n+1}\sum\limits_{p=n+1}^{n+m}(e_p-e'_p)t^{p-n-1}.$$
But $\sum\limits_{k=0}^m{m\brack k}_w t^k$ and $t^{n+1}$ are coprime, so $\sum\limits_{k=0}^m{m\brack k}_w t^k$ divides $\sum\limits_{p=n+1}^{n+m}(e_p-e'_p)t^{p-n-1}$ which has degree less than $m$ and hence must be 0, i.e. $e_p=e'_p$ for $p>n$.
\end{proof}
Using this result we may compute the coefficients $e(p,0)$ and $e(0,q)$ for any $a_1$ and $a_2$.
\begin{corollary}\label{cor:baseline}
 For any $(a_1,a_2)\in\ZZ^2$ the baseline greedy coefficients can be computed as
 \[e(p,0)={[a_2]_+\brack p}_{v^b}\quad\text{and}\quad e(0,q)={[a_1]_+\brack q}_{v^c}\quad\text{for all $p,q\ge0$.}\]
\end{corollary}
\begin{proof}
 If $a_1\le0$, then for $p=0$ the first recurrence in \eqref{eq:recurrence} always applies and, since this summation is empty, we see that $e(0,q)=0={0\brack q}_{v^c}$ for $q>0$.  Since $e(0,0)=1={0\brack0}_{v^c}$, the equality $e(0,q)={[a_1]_+\brack q}_{v^c}$ holds in this case.  A similar argument establishes the equality $e(p,0)={[a_2]_+\brack p}_{v^b}$ when $a_2\le0$.
 
 Now assume $a_2>0$.  Then $ca_1q\le ba_2p$ always holds for $q=0$ so in computing $e(p,0)$ only the first recurrence in \eqref{eq:recurrence} applies.  To prove the equality $e(p,0)={a_{2}\brack p}_{v^{b}}$ we apply Proposition~\ref{gen_recursion} with $m=a_2$, $n=0$, $w=v^b$, and $e_p=e(p,0)$.  From part (a) we obtain $e(p,0)=0$ for $p>a_2$.  Now, since $e_0=e(0,0)=1$, part (b) gives $\sum\limits_{p=0}^{a_2}e_pt^p=\sum\limits_{p=0}^{a_2}{a_2\brack p}_{v^b}t^p$ and so $e(p,0)=e_p={a_2\brack p}_{v^b}$.  A similar argument proves $e(0,q)={a_1\brack q}_{v^c}$ for $a_1>0$.
\end{proof}

For certain classes of $(a_1,a_2)\in\ZZ$ the recursion \eqref{eq:recurrence} can be computed very explicitly.

\begin{proposition}\label{prop:a1 a2 not positive}
Let $(a_1,a_2)\in\ZZ$ and define $e(p,q)$ as in \eqref{eq:recurrence}.  We have the following:

{\rm(1)} If $a_{1}, a_{2}\le0$, then $e(p,q)=0$ for $(p,q)\neq(0,0)$.

{\rm(2)} If $a_{1}\le 0<a_{2}$, then $e(p,0)={a_{2}\brack p}_{v^{b}}$ and $e(p,q)=0$ for $q>0$.

{\rm(3)}  If $a_{2}\le 0<a_{1}$, then $e(0,q)={a_{1}\brack q}_{v^{c}}$ and $e(p,q)=0$ for $p>0$.

{\rm(4)} If $0<ba_2\leq a_1$, then $e(p,q)={a_2\brack p}_{v^b}{a_1-bp\brack q}_{v^c}$.

{\rm(5)} If $0<ca_1\leq a_2$, then $e(p,q)={a_2-cq\brack p}_{v^b}{a_1\brack q}_{v^c}$.
\end{proposition}
\begin{proof}
 In Corollary~\ref{cor:baseline} we have already established all of the desired formulas for $e(p,0)$ and $e(0,q)$.  

For $a_{1}, a_{2}\le0$ both summations in \eqref{eq:recurrence} are empty and thus $e(p,q)=0$ for $(p,q)\ne(0,0)$.  For $a_{1}\le 0<a_{2}$, Corollary~\ref{cor:baseline} gives $e(0,q)=0$ for $q>0$ and then a simple induction on $p$ shows that $e(p,q)=0$ whenever $q>0$.  For $a_{2}\le 0<a_{1}$, we may similarly conclude that $e(p,q)=0$ whenever $p>0$.

Suppose $0<ba_2\leq a_1$.  We aim to show that $e(p,q)={a_2\brack p}_{v^b}{a_1-bp\brack q}_{v^c}$ satisfies \eqref{eq:recurrence}.  Indeed, substituting this expression into \eqref{eq:recurrence} gives
\begin{equation}\label{eq:case 4 recurrence}
{a_2\brack p}_{v^b}{a_1-bp\brack q}_{v^c} =  \begin{cases}
          \sum\limits_{k=1}^p(-1)^{k-1}{a_2\brack p-k}_{v^b}{a_1-b(p-k)\brack q}_{v^c}{[a_2-cq]_++k-1\brack k}_{v^b} & \text{ if }ca_1q\le ba_2p;\\
          \sum\limits_{\ell=1}^q(-1)^{\ell-1}{a_2\brack p}_{v^b}{a_1-bp\brack q-\ell}_{v^c}{[a_1-bp]_++\ell-1\brack \ell}_{v^c} & \text{ if }ca_1q\ge ba_2p.\\
         \end{cases}
\end{equation}
For $ca_1q\ge ba_2p$ this can be shown directly:
\begin{align*}
 &\sum\limits_{\ell=1}^q(-1)^{\ell-1}{a_2\brack p}_{v^b}{a_1-bp\brack q-\ell}_{v^c}{a_1-bp+\ell-1\brack \ell}_{v^c}\\
 &\quad\quad\quad\quad=-{a_2\brack p}_{v^b}\sum\limits_{\ell=1}^q{a_1-bp\brack q-\ell}_{v^c}{bp-a_1\brack \ell}_{v^c}={a_2\brack p}_{v^b}{a_1-bp\brack q}_{v^c},
\end{align*}
where the last equality follows from Corollary~\ref{cor:quantum convolution} with $m=a_1-bp$ and $n=-m$.  

However, for $ca_1q\le ba_2p$ it appears to be quite mysterious that these quantities coincide: we were unable to pass directly from one to the other as we did in the preceding case.  We will leave this as a challenge for the eager reader and proceed by another method, that is we explicitly compute an element of $\cA_v(b,c)$ having these coefficients and utilize the symmetries introduced in Section~\ref{sec:symmetries}.

As in \eqref{eq:variable power} we have
\begin{align*}
 X_3^n=\sum\limits_{k=0}^n{n\brack k}_{v^c}X^{(-n,ck)}\quad\text{and}\quad X_4^n=\sum\limits_{p=0}^n\sum\limits_{\ell=0}^{bp}{n\brack p}_{v^b}{bp\brack \ell}_{v^c}X^{(-bp,-n+c\ell)}
\end{align*}
and thus we may expand $X_3^{(a_1-ba_2,a_2)}$ as follows:
\begin{align*}
 &v^{a_2(a_1-ba_2)}\sum\limits_{k=0}^{a_1-ba_2}\sum\limits_{p=0}^{a_2}\sum\limits_{\ell=0}^{bp}{a_1-ba_2\brack k}_{v^c}{a_2\brack p}_{v^b}{bp\brack \ell}_{v^c}X^{(-a_1+ba_2,ck)}X^{(-bp,-a_2+c\ell)}\\
 &=\sum\limits_{p=0}^{a_2}\sum\limits_{k=0}^{a_1-ba_2}\sum\limits_{\ell=0}^{bp}{a_2\brack p}_{v^b}{a_1-ba_2\brack k}_{v^c}{bp\brack \ell}_{v^c}v^{c\ell(a_1-ba_2)-bckp}X^{(-a_1+b(a_2-p),-a_2+c(k+\ell))}\\
 &=\sum\limits_{p=0}^{a_2}\sum\limits_{q=0}^{a_1-b(a_2-p)}{a_2\brack p}_{v^b}\left(\sum\limits_{k+\ell=q}v^{c\ell(a_1-ba_2)-ckbp}{a_1-ba_2\brack k}_{v^c}{bp\brack \ell}_{v^c}\right)X^{(-a_1+b(a_2-p),-a_2+cq)}\\
 &=\sum\limits_{p=0}^{a_2}\sum\limits_{q=0}^{a_1-b(a_2-p)}{a_2\brack p}_{v^b}{a_1-b(a_2-p)\brack q}_{v^c}X^{(-a_1+b(a_2-p),-a_2+cq)}\\
 &=\sum\limits_{p=0}^{a_2}\sum\limits_{q=0}^{a_1-bp}{a_2\brack p}_{v^b}{a_1-bp\brack q}_{v^c}X^{(-a_1+bp,-a_2+cq)},
\end{align*}
where the third equality used Lemma \ref{cor:quantum convolution}.  Applying the symmetry $\sigma_1$ to this element we get $\sigma_1\big(X_3^{(a_1-ba_2,a_2)}\big)=X_{-2}^{(a_2,a_1-ba_2)}=\sum\limits_{p,q\ge0}d'(p,q')X^{(-a_1+bp,a_2-ca_1+cq')}$ whose pointed support region is $R_{\rm greedy}[a_1,ca_1-a_2]$, an example of the non-imaginary case (6) from Definition~\ref{df:PSR}.  In particular, its coefficient $d'(p,q')$ is zero for $q'+\Big(b-\frac{b(ca_1-a_2)}{ca_1}\Big)p\ge a_1$ which is equivalent to $ca_1q\le ba_2p$ for $q'=a_1-q$.  But notice that, as in \eqref{eq:d'}, the coefficient $d'(p,q')$ is exactly $\sum\limits_{k=0}^p(-1)^k{a_2\brack p-k}_{v^b}{a_1-b(p-k)\brack q}_{v^c}{a_2-cq+k-1\brack k}_{v^b}$ and the claim follows.

The claim for $0<ca_1\le a_2$ is established by a similar argument involving $X_{-1}^{(a_{1},a_2-ca_{1})}$.
\end{proof}

\begin{corollary}\label{cor:explicit cluster monomials}
 
Let $(a_1,a_2)\in\ZZ$ and define $e(p,q)$ as in \eqref{eq:recurrence}.  We may compute quantum greedy elements as follows:

{\rm(1)} If $a_{1}, a_{2}\le0$, then $X[a_{1},a_{2}]=X_1^{(-a_{1},-a_{2})}$.

{\rm(2)} If $a_{1}\le 0<a_{2}$, then $X[a_{1},a_{2}]=X_{0}^{(a_{2},-a_{1})}$.

{\rm(3)}  If $a_{2}\le 0<a_{1}$, then $X[a_{1},a_{2}]=X_{2}^{(-a_{2},a_{1})}$.

{\rm(4)} If $0<ba_2\leq a_1$, then $X[a_{1},a_{2}]=X_{3}^{(a_1-ba_{2},a_{2})}$.

{\rm(5)} If $0<ca_1\leq a_2$, then $X[a_{1},a_{2}]=X_{-1}^{(a_{1},a_2-ca_{1})}$.
\end{corollary}
\begin{proof}
 The first case is immediate from Proposition~\ref{prop:a1 a2 not positive}.\\
 To see that $X[a_{1},a_{2}]=X_{0}^{(a_{2},-a_{1})}$ for $a_{1}\le 0<a_{2}$ we expand the right hand side using Lemma~\ref{le:quantum binomial theorem} and see that the pointed coefficients are given as in Proposition~\ref{prop:a1 a2 not positive}:
 \[X_{0}^{(a_{2},-a_{1})}=v^{-a_1a_2}X_0^{a_2}X_1^{-a_1}=v^{a_1a_2}X_1^{-a_1}(X_2^{-1}+v^{-b}X_1^bX_2^{-1})^{a_2}=\sum\limits_{p=0}^{a_2}{a_2\brack p}_{v^b}X^{(-a_1+bp,-a_2)}.\]
 By a similar calculation we see for $a_{2}\le 0<a_{1}$ that
 \[X_2^{(-a_2,a_1)}=\sum\limits_{q=0}^{a_1}{a_1\brack q}_{v^c}X^{(-a_1,-a_2+cq)}\]
 in accord with Proposition~\ref{prop:a1 a2 not positive}.  The claims for $0<ba_2\leq a_1$ and $0<ca_1\leq a_2$ were established in the course of proving Proposition~\ref{prop:a1 a2 not positive}.
\end{proof}

Let $Y=\sum\limits_{p,q\ge0} d(p,q)X^{(-a_1+bp,-a_2+cq)}$ be an element in $\cT$ or $\cT\otimes_\ZZ\QQ$ pointed at $(a_1,a_2)$ and $R$ be a region in $\mathbb{R}^2$. We say that $Y$ satisfies the {\it pointed support condition  for} $R$ if
\begin{equation}\label{eq:support} 
\textrm{  $d(p,q)=0$\; if $(p,q)\notin R$.}
\end{equation}
The following lemma gives a non-recursive characterization of quantum greedy elements.

\begin{proposition}\label{prop:equivalent}
Theorem~\ref{main theorem1} is equivalent to the following statement:
for any $(a_1,a_2)\in\ZZ$, there exists a unique element pointed at $(a_1,a_2)$ that satisfies the divisibility condition \eqref{eq:divisibility} and the pointed support condition \eqref{eq:support} for $R_{\text{greedy}}[a_1,a_2]$.
\end{proposition}
\begin{proof}
We consider the six cases of Definition \ref{df:PSR} separately. For $(a_1,a_2)$ in cases (1)-(5), the greedy element uniquely exists and is explicitly given in Corollary~\ref{cor:explicit cluster monomials}, in particular we note that it is in $\cA_v(b,c)$.  Thus by Lemma~\ref{lem:divisibility} the greedy element satisfies the divisibility condition.  We only need to show that it is the unique element pointed at $(a_1,a_2)$ that satisfies the pointed support condition for $R_{\text{greedy}}[a_1,a_2]$, but this follows directly from Proposition \ref{prop:a1 a2 not positive} and Proposition \ref{gen_recursion}.

Now take $0 < a_1 < ba_2$ and $0 < a_2 < ca_1$ and suppose $\{e(p,q)\}$ satisfies $(\le,\ge)$. 
For $q>0$ and $p>\lceil ca_{1}q/(ba_{2})\rceil-1$ the sequence $\{e(p,q)\}_{p\ge0}$ is determined by the first recurrence relation of $(\le,\ge)$.  So for $0<q< a_2/c$ we apply Proposition \ref{gen_recursion} with $m=a_{2}-cq$, $n=\lceil ca_{1}q/(ba_{2})\rceil-1$, $w=v^{b}$ to get

(i) $e(p,q)=0$ for $p>m+n$\,, and

(ii) $\sum\limits_{k=0}^{a_2-cq}{a_2-cq\brack k}_{v^b} t^k\ \Big|\ \sum\limits_{p\ge0} e(p,q)t^p$ holds for $0<q< a_2/c$.

Combining (ii) and its symmetric counterpart we see that the divisibility condition \eqref{eq:divisibility} holds. Meanwhile, the inequality $p>m+n$ is equivalent to $p\ge a_{2}-cq+ca_{1}q/(ba_{2})$.  So (i) is equivalent to the condition that $e(p,q)=0$ if $p\ge a_{2}-cq+ca_{1}q/(ba_{2})$ and $0<q< a_2/c$.  By the symmetric argument we see that $e(p,q)=0$ if $q\ge a_1-bp+ba_2p/(ca_1)$ and $0<p< a_1/b$.  Combining these observations with Corollary \ref{cor:baseline}, we see that $e(p,q)=0$ if $(p,q)\notin OABC$. This proves  \eqref{eq:support} for $R=OABC$. The above argument can be easily reversed to complete the proof of the claim.

Next, we claim that there is at most one element satisfying \eqref{eq:divisibility} and  \eqref{eq:support} for $R=OABC$.  It suffices to show that for $(p,q)\in OABC$, $e(p,q)$ is determined by those $e(i,j)$ with $(i,j)\le(p,q)$ and $(i,j)\neq(p,q)$. To verify this,  assume $ba_2p\ge ca_1q$ without loss of generality. Then (i)(ii) hold. Proposition \ref{gen_recursion} implies that $e(p,q)$ is determined by  those $e(i,q)$ with $i\le n=\lceil ca_{1}q/(ba_{2})\rceil-1$. Such integers $i$ satisfy $(i,q)\le (p,q)$ and $(i,q)\neq (p,q)$.

The above two claims immediately conclude case (6).
\end{proof}

\begin{proof}[Proof of Corollary~\ref{cor:greedy in cluster algebra}]
 Since each quantum greedy element satisfies the divisibility condition \eqref{eq:divisibility}, Lemma~\ref{lem:divisibility} shows that it is in $\cA_v(b,c)$.
\end{proof}

\section{Definitions of upper and lower quasi-greedy elements}

In this section we define and study two variations of the quantum greedy elements, $\UX[a_1,a_2]$ and $\LX[a_1,a_2]$. 

\begin{theorem}\label{recursive-definition_u}
For each $(a_{1},a_{2})\in\ZZ^{2}$, there exists a unique element in the quantum torus $\cT$ pointed at $(a_1,a_2)$ whose coefficients $\uc(p,q)$ (resp. $\lc(p,q)$) satisfy the  recurrence relation $(\le,>)$ (resp. $(<,\ge)$) in Definition \ref{le,ge}.
 \end{theorem}
\begin{proof}
We prove the claim for $\uc(p,q)$, the proof for $\lc(p,q)$ is similar.  If $a_1<0$ or $a_2<0$, then $\uc(p,q)=e(p,q)$ as given in Proposition \ref{prop:a1 a2 not positive}.  Thus for the remainder of the proof we assume $a_1,a_2>0$.

For $p\ge a_{1}/b$ and $q\ge a_{2}/c$ we use the fact that ${k-1\brack k}_{w}=0$ for all $k\ge 1$ to conclude that $\uc(p,q)=0$. If $0\le p< a_{1}/b$, then    
Proposition~\ref{gen_recursion} implies that $\uc(p,q)=0$ for $q>a_1-(b-\frac{ba_2}{ca_1})p$, that is, when $(p,q)$ is above $BC$. Similarly, if $0\le q<a_2/c$, then $\uc(p,q)=0$ when $(p,q)$ is on or to the right of $AB$. For $q=0$, we have $\uc(p,0)={a_2\brack p}_{v^b}$. Thus $\uc(p,q)=0$ outside the region $\overline{OABC}$.
In particular, all but finitely many $\uc(p,q)$ are 0. This proves the existence.
The uniqueness is obvious. 
\end{proof}


\begin{definition}\label{df:X}
 Define the {\it upper quasi-greedy element} $\UX[a_1,a_2]$ and {\it lower quasi-greedy element} $\LX[a_1,a_2]$ in $\cT$ as
$$\UX[a_1,a_2]=\sum_{(p,q)}\uc(p,q) X^{(-a_1+bp,-a_2+cq)}, \quad
\LX[a_1,a_2]=\sum_{(p,q)}\lc(p,q) X^{(-a_1+bp,-a_2+cq)}.$$
Define $\ulc(p,q)=(\uc(p,q)+\lc(p,q))/2$ and define the {\it quasi-greedy element} in $\cT\otimes_\ZZ\mathbb{Q}$ as
$$\ULX[a_1,a_2]=(\UX[a_1,a_2]+\LX[a_1,a_2])/2=\sum_{(p,q)}\ulc(p,q) X^{(-a_1+bp, -a_2+cq)}.$$
\end{definition}

\begin{corollary}
 For $(a_1,a_2)$ from cases {\rm(1)--\rm(5)} of Definition~\ref{df:PSR} we have 
 \[\UX[a_1,a_2]=\LX[a_1,a_2]=\ULX[a_1,a_2]=X[a_1,a_2].\]
\end{corollary}
\begin{proof}
 This is immediate from Corollary~\ref{cor:explicit cluster monomials}.
\end{proof}

\subsection{Characterization of various quasi-greedy elements by axioms}
Now we give a non-recursive characterization of $\UX[a_1,a_2]$, $\LX[a_1,a_2]$ and $\ULX[a_1,a_2]$  analogous to those given in Proposition~\ref{prop:equivalent}.

First we define support regions $\overline{R}_{\text{greedy}}$, $\underline{R}_{\text{greedy}}$ and $\overline{\underline{R}}_{\text{greedy}}[a_{1},a_{2}]$.  Outside case (6) in Definition \ref{df:PSR}, we set 
$$\overline{R}_{\text{greedy}}[a_1,a_2]=\underline{R}_{\text{greedy}}[a_1,a_2]=\overline{\underline{R}}_{\text{greedy}}[a_1,a_2]=R_{\text{greedy}}[a_1,a_2],$$
i.e. for all $(a_1,a_2)$ except when $0 < a_1 < ba_2$ and $0 < a_2 < ca_1$.  In the final case, we define
$$\aligned
&\overline{R}_{\text{greedy}}[a_{1},a_{2}]={R}_{\text{greedy}}[a_{1},a_{2}]\cup\bigg\{(p,q)\in\mathbb{R}_{\ge0}^2\bigg|\; 0\le p< \frac{a_1}{b},\; q+\Big(b-\frac{ba_2}{ca_1}\Big)p= a_1\bigg\};\\
&\underline{R}_{\text{greedy}}[a_{1},a_{2}]={R}_{\text{greedy}}[a_{1},a_{2}]\cup\bigg\{(p,q)\in\mathbb{R}_{\ge0}^2\bigg|\; 0\le q< \frac{a_2}{c},\; p+\Big(c-\frac{ca_1}{ba_2}\Big)q= a_2\bigg\};\\
&\overline{\underline{R}}_{\text{greedy}}[a_{1},a_{2}]=\overline{R}_{\text{greedy}}[a_{1},a_{2}]\cup\underline{R}_{\text{greedy}}[a_{1},a_{2}].
\endaligned
$$
In other words, $\overline{R}_{\text{greedy}}[a_{1},a_{2}]=\overline{OABC}$, the region that excludes the interior of $AB$ and the point $B$ but includes the rest of the boundary;  $\underline{R}_{\text{greedy}}[a_{1},a_{2}]=\underline{OABC}$, the region that excludes the interior of $BC$ and the point $B$ but includes the rest of the boundary; $\overline{\underline{R}}_{\text{greedy}}[a_1,a_2]=\overline{\underline{OABC}}$, the region that contains all the boundary except the point $B$ (see Figure \ref{fig:OABC}).

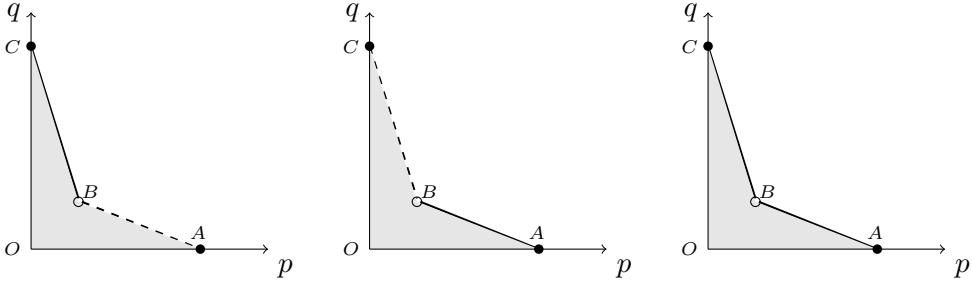
\begin{figure}[h]
\begin{tikzpicture}[scale=.9]
\begin{scope}[shift={(0,0)}]
\usetikzlibrary{patterns}
\draw[thick,dashed] (.75,.7)--(2.5,0);
\draw[thick] (0,3)--(.7,.75);
\fill [black!10]  (0,3)--(.7,.7)--(2.5,0)--(0,0)--(0,3);
\draw (0,0) node[anchor=east] {\tiny$O$};
\draw (2.2,0) node[anchor=south west] {\tiny$A$};
\draw (.6,.6) node[anchor=south west] {\tiny$B$};
\draw (0,3) node[anchor=east] {\tiny$C$};
\draw[->] (0,0) -- (3.5,0)
node[below right] {$p$};
\draw[->] (0,0) -- (0,3.5)
node[left] {$q$};
\draw (.7,.7) circle (2pt);
\fill (0,3) circle (2pt);
\fill (2.5,0) circle (2pt);
\end{scope}
\begin{scope}[shift={(5,0)}]
\usetikzlibrary{patterns}
\draw [thick] (.75,.7)--(2.5,0);
\draw[thick, dashed] (0,3)--(.7,.75);
\fill [black!10]  (0,3)--(.7,.7)--(2.5,0)--(0,0)--(0,3);
\draw (0,0) node[anchor=east] {\tiny$O$};
\draw (2.2,0) node[anchor=south west] {\tiny$A$};
\draw (.6,.6) node[anchor=south west] {\tiny$B$};
\draw (0,3) node[anchor=east] {\tiny$C$};
\draw[->] (0,0) -- (3.5,0)
node[below right] {$p$};
\draw[->] (0,0) -- (0,3.5)
node[left] {$q$};
\draw (.7,.7) circle (2pt);
\fill (0,3) circle (2pt);
\fill (2.5,0) circle (2pt);
\end{scope}
\begin{scope}[shift={(10,0)}]
\usetikzlibrary{patterns}
\draw [thick] (.75,.7)--(2.5,0);
\draw[thick] (0,3)--(.7,.75);
\fill [black!10]  (0,3)--(.7,.7)--(2.5,0)--(0,0)--(0,3);
\draw (0,0) node[anchor=east] {\tiny$O$};
\draw (2.2,0) node[anchor=south west] {\tiny$A$};
\draw (.6,.6) node[anchor=south west] {\tiny$B$};
\draw (0,3) node[anchor=east] {\tiny$C$};
\draw[->] (0,0) -- (3.5,0)
node[below right] {$p$};
\draw[->] (0,0) -- (0,3.5)
node[left] {$q$};
\draw (.7,.7) circle (2pt);
\fill (0,3) circle (2pt);
\fill (2.5,0) circle (2pt);
\end{scope}
\end{tikzpicture}
\caption{Left: $\overline{OABC}$,\quad  Center: $\underline{OABC}$,\quad Right: $\overline{\underline{OABC}}$}
\label{fig:OABC}
\end{figure}

\begin{proposition}\label{quasi-greedy_axiom}
Let $(a_1,a_2)\in\ZZ$.

{\rm(1)} $\UX[a_1,a_2]$ (resp. $\LX[a_1,a_2]$) is the unique element in the quantum torus $\cT$
 pointed at $(a_1,a_2)$ that satisfies the divisibility condition \eqref{eq:divisibility} and the support condition \eqref{eq:support} for $\overline{R}_{\text{greedy}}[a_1,a_2]$ (resp. for $\underline{R}_{\text{greedy}}[a_1,a_2]$).

{\rm(2)} $\ULX[a_1,a_2]\in \cT\otimes_\ZZ\QQ$ is pointed at $(a_1,a_2)$ and satisfies the divisibility condition \eqref{eq:divisibility} and the support condition \eqref{eq:support} for $\overline{\underline{R}}_{\text{greedy}}[a_1,a_2]$.
  
As a consequence, $\UX[a_1,a_2]$ and $\LX[a_1,a_2]$ are in $\myAA_v(b,c)$ and $\ULX[a_1,a_2]$ is in $\myAA_v(b,c)\otimes_\ZZ\QQ$.
\end{proposition}
\begin{proof}
The proof of (1) is similar to Proposition~\ref{prop:equivalent}.
(2) follows from (1). The consequence follows from Lemma \ref{lem:divisibility}.
\end{proof}

Next we show that upper and lower quasi-greedy elements behave nicely under automorphisms $\sigma_1$ and $\sigma_2$ of $\myAA_v(b,c)$. 

\begin{proposition}\label{prop:mutation invariant:q}
The automorphisms $\sigma_1$ and
$\sigma_2$ act on the  upper and lower quasi-greedy elements as follows:
$$\aligned
&\sigma_1(\UX[a_1,a_2])=\LX[a_1,c[a_1]_+-a_2],\quad \sigma_1(\LX[a_1,a_2])=\UX[a_1,c[a_1]_+-a_2],\\
&\sigma_2(\UX[a_1,a_2])=\LX[b[a_2]_+-a_1,a_2],\quad \sigma_2(\LX[a_1,a_2])=\UX[b[a_2]_+-a_1,a_2].
\endaligned
$$
\end{proposition}
\begin{proof}
Since the proof of the four identities are similar, 
we only prove the third identity 
$$\sigma_2(\UX[a_1,a_2])=\LX[b[a_2]_+-a_1,a_2].$$
In cases (1)--(5) in Definition \ref{df:PSR}, the proof is straightforward using Proposition \ref{prop:a1 a2 not positive}.
For example in case (1),
$$\aligned
\sigma_2(\UX[a_1,a_2])&=\sigma_2(X^{(-a_1,-a_2)})=v^{-a_1a_2}X_3^{-a_1}X_2^{-a_2}=v^{a_1a_2}X_2^{-a_2}X_3^{-a_1}\\
&=X_2^{(-a_2,-a_1)}=\LX[-a_1,a_2]=\LX[b[a_2]_+-a_1,a_2].\\
\endaligned
$$
So in the rest of the proof we focus on cases (6), that is, $0 < a_1 < ba_2$ and $0 < a_2 < ca_1$.

Write $a'_1=ba_2-a_1$ and define 
$$Z:=\sigma_2(\UX[a_1,a_2])=\sum_{p,q\ge0} d'(p,q) X^{(bp-a'_1,cq-a_2)}$$
It follows from Proposition~\ref{quasi-greedy_axiom} and Lemma~\ref{lem:divisibility} that $\UX[a_1,a_2]\in\cA_v(b,c)$ and thus $Z$ is an element in $\myAA_v(b,c)$ since the quantum cluster algebra is stable under all symmetries.  By the same reasoning both $\sigma_1(Z)$ and $\sigma_2(Z)$ are in $\cA_v(b,c)$ and it then follows from Lemma~\ref{lem:divisibility} that $Z$ satisfies the divisibility condition \eqref{eq:divisibility}.  To prove $\LX[a'_1,a_2]=Z$ it only remains to verify that $Z$ satisfies the support condition \eqref{eq:support} for $\underline{R}_{\text{greedy}}[a_1',a_2]$.

 The same argument as in the proof of Corollary~\ref{cor:baseline} computes $\uc(p,q)$ and $\lc(p,q)$ when $p=0$ or $q=0$, so we focus on the interesting case when $p>0$. It suffices to show that

(i) if $1\le p\le a'_1/b$, then $d'(p,q)=0$ for $q\ge a'_1-\Big(b-\displaystyle\frac{ba_2}{ca'_1}\Big)p$;

(ii) if $a'_1/b<p<a_2$, then $d'(p,q)=0$ for $q>(a_2-p)\Big/\Big(c-\displaystyle\frac{ca'_1}{ba_2}\Big)$.\\

\noindent To prove (i) we fix $p$ and compute the degree of the polynomial $\sum_{q\ge0} d'(p,q)t^q$.  Indeed, using \eqref{eq:d'} and the support condition for $\UX[a_1,a_2]$ we get
$$\aligned
\deg\Big(\sum_{q\ge0} d'(p,q)t^q\Big)&=\deg\Big(\sum_{q\ge0}\uc(a_2-p,q)t^q\Big)-\big(a_1-b(a_2-p)\big)\\
&=\deg \Big(\sum_{q\ge0}\uc(a_2-p,q)t^q\Big)+(a'_1-bp)\\
&<p\Big/\Big(c-\frac{ca_1}{ba_2}\Big)+(a'_1-bp)\\
&=\frac{ba_2}{ca'_1}p+a'_1-bp=a'_1-\Big(b-\frac{ba_2}{ca'_1}\Big)p.
\endaligned
$$
\noindent
The proof of (ii) is similar and we leave it as an exercise for the reader.
\end{proof}

Recall that positive imaginary roots are lattice points in the set 
\[\Phi^{im}_+:=\Big\{(a_1,a_2)\in\mathbb{Z}_{>0}^2: ca_1^2-bca_1a_2+ba_2^2\le 0\Big\}.\]
The real roots are lattice points in $\ZZ^2\setminus \Phi^{im}_+$.
It is well-known that the set of denominator vectors of all cluster monomials is exactly $\ZZ^2\setminus\Phi^{im}_+$.

\begin{corollary}\label{cor:cluster monomials:q}
 Let $k$ be an integer and write $P_k,P_{k+1}\in\ZZ^2$ for the denominator vectors of $x_k$ and $x_{k+1}$ respectively. Then for $(a_1,a_2)=mP_k+nP_{k+1}$ with $m,n\in\mathbb{Z}_{\ge0}$, we have
$$\UX[a_1,a_2]=\LX[a_1,a_2]=X[a_1,a_2]=X_k^{(m,n)}.$$
\end{corollary}
\begin{proof}
The statement obviously holds for $k=1$.
All other cluster monomials can be obtained from $X_1^{(m,n)}$ by iteratively applying $\sigma_1$ and $\sigma_2$. Therefore the statement follows from the $k=1$ case and Proposition \ref{prop:mutation invariant:q}.
\end{proof}

\section{$\ULX[a_1,a_2]$: basic properties}

Here we study the quasi-greedy elements $\ULX[a_1,a_2]$ introduced in Definition \ref{df:X}.

\begin{proposition}\label{mean invariant:q}
The quasi-greedy elements form a mutation invariant $\mathbb{Q}[v^{\pm1}]$-basis of the cluster algebra $\mathcal{A}_v(b,c)\otimes_\mathbb{Z}\mathbb{Q}$. The automorphisms $\sigma_1$ and
$\sigma_2$ act on quasi-greedy elements as follows:
$$\sigma_1(\ULX[a_1,a_2])=\ULX[a_1,c[a_1]_+-a_2],\quad \sigma_2(\ULX[a_1,a_2])=\ULX[b[a_2]_+-a_1,   a_2].$$
\end{proposition}
\begin{proof}
The second statement follows immediately from Proposition \ref{prop:mutation invariant:q}.
The proof of the first statement follows an argument similar to that in \S6 of \cite{llz} by comparing the family of quasi-greedy elements with the
standard monomial basis and showing that the transition matrix is invertible, so we omit the proof.
\end{proof}

Now we introduce the notion of special region which will be used in Lemma \ref{lemma:support:q}.

\begin{definition}\label{df:special region}
We call $R\subseteq\mathbb{R}^2$ a {\it special region} if it is the closed region bounded by a polygon $R_0R_1\cdots R_n$, where  $R_i=(u_i,v_i)$ for $i=1,\dots, n$ with
$$u_1\le u_2\le \cdots \le u_n,\quad v_1\ge v_2\ge \cdots\ge v_n,$$
$R_0=(u_1,v_n)$,
and if $R$ is \emph{origin convex}, i.e. $R$ satisfies the property that for any point $p\in R$, the whole line segment connecting the origin $(0,0)$ and $p$ is in $R$.
\end{definition}

\begin{remark}
 It follows from the definition of a special region that $u_1\le 0\le u_n$ and $v_1\ge 0\ge v_n$.  Thus any special region can be pictured as in Figure \ref{fig:special region}.
\end{remark}

\begin{figure}[h]
\begin{tikzpicture}[scale=.8]
\begin{scope}[shift={(11,0)}]
\usetikzlibrary{patterns}
\draw[black] (0,5)--(1,3.5)--(2,3.2)--(5,.5)--(6.5,0)--(0,0)--(0,5);
\fill [black!10]  (0,5)--(1,3.5)--(2,3.2)--(5,.5)--(6.5,0)--(0,0)--(0,5);
\fill (0,0) circle (1pt);
\fill (0,5) circle (1pt);
\fill (1,3.5) circle (1pt);
\fill (2,3.2) circle (1pt);
\fill (5,.5) circle (1pt);
\fill (6.5,0) circle (1pt);
\fill (1.5,1) circle (1pt);
\draw[black,dashed] (1.5,1)--(0,0) (1.5,1)--(0,5) (1.5,1)--(1,3.5) (1.5,1)--(2,3.2) (1.5,1)--(5,.5) (1.5,1)--(6.5,0);
\draw (0,0) node[anchor=east] {\tiny$R_0$};
\draw (0,5) node[anchor=east] {\tiny$R_1$};
\draw (0.9,3.7) node[anchor=west] {\tiny$R_2$};
\draw (2,3.3) node[anchor=west] {\tiny$R_3$};
\draw (3.1,2.3) node[anchor=west] {\tiny$\ddots$};
\draw (3.5,1.93) node[anchor=west] {\tiny$\ddots$};
\draw (4.8,.8) node[anchor=west] {\tiny$R_{n-1}$};
\draw (6.5,0) node[anchor=west] {\tiny$R_n$};
\draw (1.5,1) node[anchor=east] {\tiny$(0,0)$};
\end{scope}
\end{tikzpicture}
\caption{A special region $R$}
\label{fig:special region}
\end{figure}
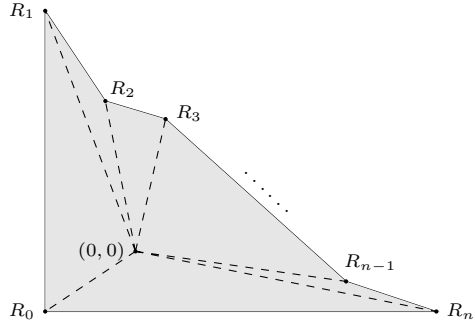

Define $\overline{\underline{S}}[a_{1},a_{2}]=\varphi(\overline{\underline{R}}_{\textrm{greedy}}[a_{1},a_{2}])$. Note that it is a closed region in cases (1)--(5) of Definition \ref{df:PSR}, while in case (6) it is obtained by removing the origin from the boundary of a closed region.
 
\begin{lemma}\label{lemma:support:q}\mbox{}
\begin{enumerate}[\upshape (i)]
 \item If $\overline{\underline{S}}[a_1,a_2]$ contains a point $(u,v)\in\RR_{\ge0}^2$, then 
$$a_1=-u,\quad a_2=-v, \quad  
\overline{\underline{S}}[a_1,a_2]=\{(u,v)\}.$$
 \item If $(a_1,a_2)\neq(a_1',a_2')$ and $\overline{\underline{S}}[a_1,a_2]\subseteq\overline{\underline{S}}[a_1',a_2']$, then 
$$0 < a_1 < ba_2, \; 0 < a_2 < ca_1, \;0 < a_1' < ba_2', \; 0 < a_2' < ca_1', \; \textrm{ and }\; a_1:a_2=a_1':a_2'.$$
 \item Let $Z$ be an element in the quantum cluster algebra $\mathcal{A}_v(b,c)$ with the linear expansion
$$Z=\sum_{a_1,a_2}d_{a_1,a_2}\ULX[a_1,a_2], \quad \textrm{ where } d_{a_1,a_2}\in\mathbb{Q}[v^{\pm1}].$$
 If the support of $Z$  is contained in a special region $R$,
then the support region $\overline{\underline{S}}[a_{1},a_{2}]$ of any quasi-greedy element $\ULX[a_1,a_2]$ corresponding to a nonzero $d_{a_1,a_2}$ is contained in $R$.
\end{enumerate}
\end{lemma}

\begin{proof}
(i) and (ii) follow from an easy case-by-case study of the six cases in Definition \ref{df:PSR} and we skip the proof. (It is helpful to look at Figure \ref{fig:PSR} and observe that $\varphi(B)=(0,0)$.)

(iii) Aiming at a contradiction we assume that there are integers $a_1,a_2$ with $d_{a_1,a_2}\neq0$ such that there is a point $(u,v)\in\overline{\underline{S}}[a_1,a_2]$ with $(u,v)\notin R$. Without loss of generality, we may assume that there is no other $d_{a_1',a_2'}\neq0$ with $-a'_1\le -a_1$ and $-a'_2\le -a_2$.  Then, since each is pointed, the monomial $X^{(-a_1,-a_2)}$ of $\ULX[a_1,a_2]$ cannot appear in any other $\ULX[a'_1,a'_2]$ and the point $(-a_1,-a_2)$ is in the support of $Z$, thus the point $(-a_1,-a_2)$ is in $R$.  It follows that
\begin{equation}\label{R_0}
(-a_1,-a_2)\ge (u_1,v_n)=R_0.
\end{equation} 

Since $(-a_1,-a_2)\le(u,v)$, we see that for $u<u_1$ or $v<v_n$ we have $(-a_1,-a_2)\notin R$, a contradiction.  In particular, $u,v\le0$ implies $u<u_1$ or $v<v_n$, which we have seen is impossible.  If $u,v\ge0$, then (i) asserts that $(-a_1,-a_2)=(u,v)\notin R$, again a contradiction.  Thus we must have $u_1<u<0$ and $v>0$ or $u>0$ and $v_1<v<0$.  In the rest we assume without loss of generality that $u_1<u<0$ and $v>0$, this is only possible in cases (3)(4)(6) of Definition \ref{df:PSR}.

  We assert that $(-a_1,-a_2+ca_1)\notin R$. To show this, we consider each case separately:
\smallskip

\noindent Case (3): $a_2 \leq 0 < a_1$:  The support region $\overline{\underline{S}}[a_1,a_2]$ is the vertical segment $\varphi(OC)$ connecting $\varphi(O)=(-a_1,-a_2)$ and $\varphi({C})=(-a_1,-a_2+ca_1)$. Since both $\varphi(O)$ and $\varphi({C})$ are weakly to the northeast of $R_0$, from the shape of $R$ we may conclude that $\varphi({C})\notin R$. Indeed, if $\varphi({C})$ is in $R$, then the whole segment $\varphi(OC)$ is in $R$, in particular $(u,v)\in\varphi(OC)$ is in $R$, a contradiction.

\smallskip

\noindent Case (4): $0<ba_2\le a_1$. The support region $\overline{\underline{S}}[a_1,a_2]$ is the trapezoid $\varphi(OAD_1C)$ which lies strictly to the west of the vertical line through $(0,0)$.  The points $\varphi(O)$ and $\varphi(A)$ are in $R$, to the southwest of $(0,0)$ while the line through $\varphi(C)$ and $\varphi(D_1)$ passes below $(0,0)$, intersecting the vertical line through $(0,0)$ at $(0,-a_2)$ (see Figure \ref{fig:PSR}).  By the shape of the special region $R$ we see that if $\varphi({C})$ is in $R$, then the whole segment $\varphi(OC)$ is in $R$.  But, since $v>0$, the line through $(0,0)$ and $(u,v)$ intersects $\varphi(OC)$ at a point $P\in R$. But then $P\in R$ and $(u,v)\notin R$, contradicting the origin convexity of $R$.
\smallskip

\noindent Case (6): $0 < a_1 < ba_2$, and $0 < a_2 < ca_1$. The proof of this case is similar to Case (4) using that $\varphi(B)=(0,0)$.
\smallskip

Now consider all pairs $(a_1,a_2)$ with $d_{a_1,a_2}\neq0$ and $(-a_1,-a_2+ca_1)\notin R$ (by assumption and the above considerations there exists at least one such a pair), and take one with maximal $-a_2+ca_1$. Then the monomial $X^{(-a_1,-a_2+ca_1)}$ of $\ULX[a_1,a_2]$ cannot appear in any other $\ULX[a'_1,a'_2]$ and the point $(-a_1,-a_2+ca_1)$ is in the support of $Z$, thus the point $(-a_1,-a_2+ca_1)$ is in $R$.  This contradiction completes the proof of (iii).
\end{proof}

The next lemma studies certain coefficients in the quasi-greedy elements. Recall the region $OABC$ in case (6): $0<a_1<ba_2$ and $0<a_2<ca_1$ of Figure \ref{fig:PSR}, which we reproduce here  in Figure \ref{fig:p'q'} (for convenience we only draw one of the two figures). A lattice point $(p,q)$ is on the interior of the segment $OB$ if $0<p<a_1/b$ and $p:q=ca_{1}:ba_{2}$. It is easy to check for such a point $(p,q)$ that $(p,q+a+1-bp)$ is on the edge $BC$ while $(p+a_2-cq,q)$ is on the edge $AB$.
\begin{figure}[h]
\begin{tikzpicture}[scale=1.2]
\begin{scope}[shift={(2,0)}]
\usetikzlibrary{patterns}
\draw[thick] (0,3)--(0,0)--(2.5,0);
\draw[thick] (0,3)--(.7,.7)--(2.5,0);
\fill [black!10]  (0,3)--(.7,.7)--(2.5,0)--(0,0)--(0,3);
\draw (0,0) node[anchor=east] {\tiny$O$};
\draw (2.4,0) node[anchor=west] {\tiny$A$};
\draw (.6,.6) node[anchor=south west] {\tiny$B$};
\draw (0,3) node[anchor=east] {\tiny$C$};
\draw[thick, black!50] (0,0)--(.7,.7);
\draw[thick, black!50] (1.7,.3)--(.3,.3)--(.3,2);
\fill (.3,.3) circle (1pt);
\fill (.3,2) circle (1pt);
\fill (1.75,.3) circle (1pt);
\draw (.28,.15) node[anchor=west] {\tiny$(p,q)$};
\draw (.3,2) node[anchor=west] {\tiny$(p,q')=(p,q+a_1-bp)$};
\draw (1.8,.4) node[anchor=west] {\tiny$(p',q)=(p+a_2-cq,q)$};
\draw (-2,.4) node[anchor=west] {};
\end{scope}
\end{tikzpicture}
\caption{}
\label{fig:p'q'}
\end{figure}

\begin{lemma}\label{lem:pq':q} Assume $0<a_1<ba_2$ and $0<a_2<ca_1$. Consider a lattice point $(p,q)$ on the interior of the segment $OB$. Denote $q':=q+a_1-bp$ and $p':=p+a_2-cq$. 

{\rm(i)} The following equalities hold:
$$\ulc(p,q')=\uc(p,q')/2,\quad \ulc(p',q)=\lc(p',q)/2.$$

{\rm(ii)} Assume $\uc(i,j)=\lc(i,j)$ for every lattice point $(i,j)$ on the interior of the segment $OB$ with $i<p$.  Then $\uc(i,j)=\lc(i,j)$ for every lattice point $(i,j)\neq(p,q)$ satisfying $0\le i\le p$ and $0\le j\le q$. Moreover, 
 $$\ulc(p,q')=\frac{\uc(p,q)-\lc(p,q)}{2},\quad \ulc(p',q)=\frac{\lc(p,q)-\uc(p,q)}{2}.$$
As a consequence, if $(p,q)$ is the highest interior point in $BC$ with $\ulc(p,q)\neq0$, then 
$$\ulc(p,q)=\frac{\uc(p,q-a_{1}+bp)-\lc(p,q-a_{1}+bp)}{2}.$$
\end{lemma}

\begin{proof}
By considering the support regions of $\UX[a_1,a_2]$ and $\LX[a_1,a_2]$ we immediately see (i).

The first claim $\uc(i,j)=\lc(i,j)$ of (ii) follows from the definition of $\uc(i,j)$ and $\lc(i,j)$.  To prove the second claim of (ii), observe that  $(1+t)^{a_1-bp}$ divides
$$\sum_{i=0}^{q'}\uc(p,i)t^i-\sum_{i=0}^{q'}\lc(p,i)t^i=\sum_{i=q}^{q'}(\uc-\lc)(p,i)t^i=t^q\sum_{i=0}^{a_1-bp}(\uc-\lc)(p,q+i)t^i,$$ but the last sum has degree at most $a_1-bp$, so it is a constant multiple of $(1+t)^{a_1-bp}$. Then its coefficients of the lowest and highest degrees are equal, which gives
 $$(\uc-\lc)(p,q)=(\uc-\lc)(p,q+a_1-bp)=(\uc-\lc)(p,q')=\uc(p,q'),$$
where we have used $\lc(p,q')=0$ in the last equality.  By (i) the first equality in the second claim of (ii) follows.  The equality $\ulc(p',q)=\frac{\lc(p,q)-\uc(p,q)}{2}$ can be proved similarly.
\end{proof}

We denote $\uc(i,j)_{a_1,a_2}$ to illustrate the dependence of $\uc(i,j)$ on $(a_1,a_2)$.
The meaning of $\lc(i,j)_{a_1,a_2}$, $\lc(i,j)_{a_1,a_2}$ and $(\uc-\lc)(i,j)_{a_1,a_2}$ is similar.

\begin{lemma}\label{lem:linear}
Let $(a_{1},a_{2})$ be a positive imaginary root.  Consider a lattice point $(p,q)$ on the interior of the segment $OB$ and assume $\uc(i,j)_{a_1,a_2}=\lc(i,j)_{a_1,a_2}$ for every lattice point $(i,j)$ on the interior of the segment $OB$ with $i<p$.  Then the following are true  for every positive integer $n$:
\smallskip

{\rm(i)} for every lattice point $(i,j)$ on interior of the segment $OB$ with $i<p$, we have
$$\uc(i,j)_{na_1,na_2}=\lc(i,j)_{na_1,na_2}.$$

{\rm(ii)} $\uc(p,q)_{na_1,na_2}-\lc(p,q)_{na_1,na_2}=n(\uc(p,q)_{a_1,a_2}-\lc(p,q)_{a_1,a_2}).$
\end{lemma}
\begin{proof}
We first introduce some notation and restate (i)(ii).  For simplicity we write $X^{P}=X^{(i,j)}$ for $P=(i,j)$.  For a positive integer $k$, we describe lattice points necessary for understanding the support region of the product $\ULX[a_1,a_2]\cdot \ULX[ka_1,ka_2]$:
\begin{equation}\label{OACD}
\aligned
&O=(0,0),\\
&A_k=k\cdot(-a_1+ba_2,-a_2),\quad  C_k=k\cdot(-a_1,-a_2+ca_1),\quad  D_k=k\cdot(-a_1,-a_2),\\
&E_{k}=A_k+C_1=(kba_2-(k+1)a_1,ca_1-(k+1)a_2),\\
&F=A_1+C_1=(ba_2-2a_1,ca_1-2a_2),\\
&G_{k}=A_1+C_k=(ba_2-(k+1)a_1,kca_1-(k+1)a_2),\\
&P_k=(-ka_1+bp,-ka_2+cq),\\
&P'_k=(-ka_1+bp,-ka_2+c(q+ka_1-bp)),\\
&P''_k=(-ka_1+b(p+ka_2-cq),-ka_2+cq).
\endaligned
\end{equation}
Then the support region 
$\overline{\underline{S}}[ka_1,ka_2]$ is the region $D_kA_kOC_k\setminus\{O\}$. Note that the region is concave because $(ka_1,ka_2)$ is a positive imaginary root.
The support of the product $\ULX[a_1,a_2]\cdot \ULX[ka_1,ka_2]$ is
contained in the Minkowski sum of $D_1A_1OC_1$ and $D_kA_kOC_k$, which is
the closed region
$$R_{\rm prod}:=D_{k+1}A_{k+1}A_kE_{k}FG_{k}C_kC_{k+1}D_{k+1}.$$

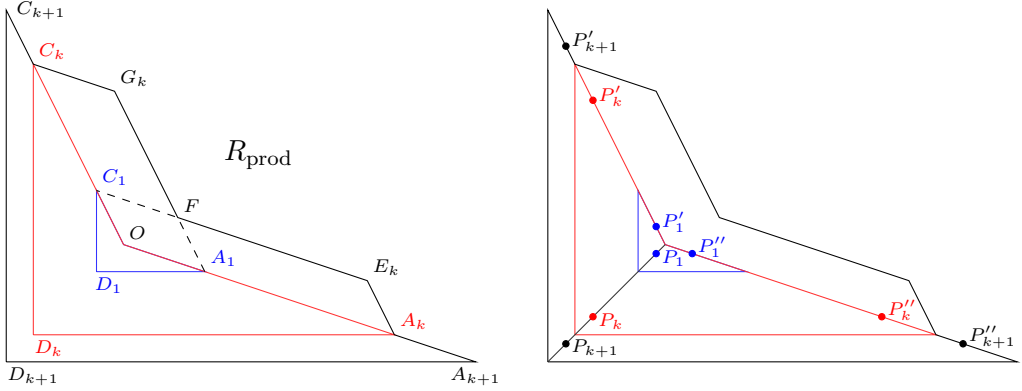
\begin{figure}[h]
\begin{tikzpicture}[scale=1.2]
\begin{scope}[shift={(0,0)}]
\usetikzlibrary{patterns}
\draw[blue!80] (-.3,-.3)--(.9,-.3)--(0,0)--(-.3,.6)--(-.3,-.3);
\draw[red!80] (-1,-1)--(3,-1)--(0,0)--(-1,2)--(-1,-1);
\draw[black] (-1.3,-1.3)--(3.9,-1.3)--(3,-1)--(2.7,-.4)--(.6,.3)--(-.1,1.7)--(-1,2)--(-1.3,2.6)--(-1.3,-1.3);
\draw[dashed] (.9,-.3)--(.6,.3)--(-.3,.6);
\draw[blue] (.1,-.45) node[anchor=east] {\tiny$D_1$};
\draw[red] (-1.1,-1.35) node[anchor=south west] {\tiny$D_k$};
\draw[blue] (.85,-.35) node[anchor=south west] {\tiny$A_1$};
\draw[blue] (-.35,.55) node[anchor=south west] {\tiny$C_1$};
\draw (-1,-1.25) node[anchor=north] {\tiny$D_{k+1}$};
\draw (3.9,-1.25) node[anchor=north] {\tiny$A_{k+1}$};
\draw[red] (2.95,-1.05) node[anchor=south west] {\tiny$A_k$};
\draw (2.65,-.45) node[anchor=south west] {\tiny$E_{k}$};
\draw (.55,.25) node[anchor=south west] {\tiny$F$};
\draw (-.15,1.65) node[anchor=south west] {\tiny$G_{k}$};
\draw[red] (-1.05,1.95) node[anchor=south west] {\tiny$C_k$};
\draw (-1.3,2.6) node[anchor=west] {\tiny$C_{k+1}$};
\draw (-.05,-.05) node[anchor=south west] {\tiny$O$};
\draw (1,1) node[anchor=west] {$R_{\rm prod}$};
\end{scope}
\begin{scope}[shift={(6,0)}]
\usetikzlibrary{patterns}
\draw[blue!80] (-.3,-.3)--(.9,-.3)--(0,0)--(-.3,.6)--(-.3,-.3);
\draw[red!80] (-1,-1)--(3,-1)--(0,0)--(-1,2)--(-1,-1);
\draw[black] (-1.3,-1.3)--(3.9,-1.3)--(3,-1)--(2.7,-.4)--(.6,.3)--(-.1,1.7)--(-1,2)--(-1.3,2.6)--(-1.3,-1.3);
\draw[black!80] (-1.3,-1.3)--(0,0);
\filldraw[black] (-1.1,-1.1) circle(1pt);\draw (-1.15,-1.15) node[anchor=west] {\tiny$P_{k+1}$};
\filldraw[black] (-1.1,2.2) circle(1pt);\draw (-1.15,2.25) node[anchor=west] {\tiny$P'_{k+1}$};
\filldraw[black] (3.3,-1.1) circle(1pt);\draw (3.25,-1.025) node[anchor=west] {\tiny$P''_{k+1}$};
\filldraw[red] (-0.8,-0.8) circle(1pt);\draw[red] (-0.85,-0.85) node[anchor=west] {\tiny$P_{k}$};
\filldraw[red] (-.8,1.6) circle(1pt);\draw[red] (-.85,1.65) node[anchor=west] {\tiny$P'_{k}$};
\filldraw[red] (2.4,-0.8) circle(1pt);\draw[red] (2.35,-0.725) node[anchor=west] {\tiny$P''_{k}$};
\filldraw[blue] (-.1,-.1) circle(1pt);\draw[blue] (-.15,-.15) node[anchor=west] {\tiny$P_{1}$};
\filldraw[blue] (-.1,0.2) circle(1pt);\draw[blue] (-.15,0.25) node[anchor=west] {\tiny$P'_{1}$};
\filldraw[blue] (0.3,-.1) circle(1pt);\draw[blue] (0.25,-.025) node[anchor=west] {\tiny$P''_{1}$};
\end{scope}
\end{tikzpicture}     
\caption{Minkowski sum $R_{\rm prod}$}
\label{fig:Minkowski sum}                               
\end{figure}
\noindent Note that by our choice of $(p,q)$ each point $P_k$ lies on the interior of the segment $OD_k$ while $P'_k$ (resp.~$P''_k$) is the intersection of the line $OC_k$ (resp.~$OA_k$) with the vertical (resp.~horizontal) line passing through $P_k$ (see Figure \ref{fig:Minkowski sum}).

For convenience, write $d^{(k)}_P$ for the coefficients in
$$\ULX[ka_1,ka_2]=\sum_{P\in\ZZ^{2}}d^{(k)}_PX^{P}.$$
In other words, $d^{(k)}_{-ka_1+bi,-ka_2+cj}=\ulc(i,j)_{ka_1,ka_2}$.  Then using Lemma \ref{lem:pq':q}(ii), we can restate conditions (i) and (ii) as follows:

{\rm(i')} for every lattice point $P$ in the interior of $C_{n}P_{n}'$, we have $d^{(n)}_{P}=0$;

{\rm(ii')} $d^{(n)}_{P'_{n}}=n\, d^{(1)}_{P'_{1}}$.

We prove (i')(ii') by induction on $n$. For $n=1$ the first holds by our assumptions and Lemma \ref{lem:pq':q}(ii), while the second is trivial. So we assume that they hold for $n\le k$.

Denote
\begin{equation}\label{eq:product Laurent:q}
\ULX[a_1,a_2]\cdot \ULX[ka_1,ka_2]=\sum_{i,j}g_{i,j}X^{(i,j)}=\sum_{i,j}h_{i,j}\ULX[i,j], 
\end{equation}
where $g_{i,j},\; h_{i,j}\in\mathbb{Q}[v^{\pm1}]$.  Note that $g_{i,j}\ne0$ implies $i=-(k+1)a_1+bs$ and $j=-(k+1)+ct$ for some $s,t\in\ZZ_{\ge0}$.

We first aim to compute the coefficient $g_{P'_{k+1}}$.   We will need the simple observation that $X^{(a,b)}\cdot X^{(c,d)}=v^{bc-ad}X^{(a+c,b+d)}$ for any integers $a,b,c,d$. In particular, if $a:b=c:d$ then 
$$X^{(a,b)}\cdot X^{(c,d)}=X^{(a+c,b+d)}=X^{(c,d)}\cdot X^{(a,b)},$$ 
in other words, $X^{(a,b)}$ and $X^{(c,d)}$ commute if the lattice points $(a,b), (c,d), (0,0)$ are collinear.  Note that by assumption $d^{(1)}_P=0$ for every lattice point $P$ on the interior of $C_1P'_1$ and $d^{(k)}_P=0$ for every lattice point $P$ on the interior of $C_kP'_k$.  It follows that $g_P=0$ for any lattice point $P$ on the interior of $C_{k+1}P'_{k+1}$ and that, to understand $g_{P'_{k+1}}$, we only need to consider two decompositions of ${P'_{k+1}}$ as the sum of a point in the support of $\ULX[ka_1,ka_2]$ and a point in the support of $\ULX[a_1,a_2]$, namely $C_1+P_k$ and $P'_1+C_k$.
Thus
$$\aligned
g_{P'_{k+1}}X^{P'_{k+1}}&=d^{(1)}_{C_{1}}X^{C_{1}}\cdot d^{(k)}_{P'_{k}}X^{P'_{k}}+d^{(1)}_{P'_{1}}X^{P'_{1}}\cdot d^{(k)}_{C_{k}}X^{C_k}\\
&=d^{(1)}_{C_{1}}d^{(k)}_{P'_{k}}X^{P'_{k+1}}+d^{(1)}_{P'_{1}}d^{(k)}_{C_{k}}X^{P'_{k+1}}
\endaligned
$$
where we used the fact that $C_{k}, P'_{k}, C_{1}, P'_{1}$ and $(0,0)$ are collinear.
Since $d^{(k)}_{C_{k}}=d^{(1)}_{C_{1}}=1$ (for example, $d^{(1)}_{C_{1}}=\ulc(0,a_1)_{a_1,a_2}={a_1\brack a_1}_{v^c}=1$), we have
\begin{equation}\label{eq:d}
g_{P'_{k+1}}=d^{(k)}_{P'_{k}}+d^{(1)}_{P'_{1}}.
\end{equation}
Meanwhile by the induction assumption and Lemma \ref{lem:pq':q}(ii), for $n\le k$ we have
\begin{equation}\label{eq:dn}
d^{(n)}_{P'_{n}}=\ulc(p,q+na_1-bp)_{na_1,na_2}=(\uc-\lc)(p,q)_{na_1,na_2}/2.
\end{equation}
So \eqref{eq:d} and \eqref{eq:dn}, together with the inductive assumption (ii'), imply
\begin{equation}\label{eq:2d P'}
2g_{P'_{k+1}}=(\uc-\lc)(p,q)_{ka_1,ka_2}+(\uc-\lc)(p,q)_{a_1,a_2}=(k+1)(\uc-\lc)(p,q)_{a_1,a_2}.
\end{equation}
Similarly,
\begin{equation}\label{eq:2d P''}
2g_{P''_{k+1}}=(\lc-\uc)(p,q)_{ka_1,ka_2}+(\lc-\uc)(p,q)_{a_1,a_2}=(k+1)(\lc-\uc)(p,q)_{a_1,a_2}.
\end{equation}

By Lemma \ref{lemma:support:q}(iii), all quasi-greedy elements that appear in the right hand side of \eqref{eq:product Laurent:q} with nonzero coefficients have their support regions lying inside $R_{\rm prod}$.  Therefore, if $h_{i,j}\neq0$ (recall that $h_{i,j}$ is defined in \eqref{eq:product Laurent:q}) and either $i\ge ka_{1}$ or $j\ge ka_{2}$, then 
$(i,j)=(\lambda a_{1},\lambda a_{2})$ where $\lambda\in\mathbb{Q}$ satisfies $k\le \lambda\le k+1$.  Indeed, write $i\ge ka_{1}$ as $(k+1)a_1-bs$ so the point of $OC_{k+1}$ above $-i$ is $\Big(-i,-(k+1)a_2+c\big(\frac{ba_2s}{ca_1}+i\big)\Big)$, where we note that $\Big(-i,-(k+1)a_2+\frac{ba_2s}{a_1}\Big)$ lies on $OD_{k+1}$.  If $j<(k+1)a_2-\frac{ba_2s}{a_1}$ (i.e. $(-i,-j)$ lies above $OD_{k+1}$), then the point $(-i,-j+ci)$ from the support of $\ULX[i,j]$ is not contained in $R_{\rm prod}$, a contradiction.  A similar argument gives the claim for $j\ge ka_{2}$.

Clearly, $g_{P_{k+1}}=1$ and so we must have $h_{(k+1)a_1,(k+1)a_2}=1$.  It follows that the support of $\ULX[a_1,a_2]\cdot \ULX[ka_1,ka_2]-\ULX[(k+1)a_1,(k+1)a_2]$ must be contained in the region obtained from $R_{\rm prod}$ by removing a strip of width 1 from the west and south boundaries.  Since $g_P=0$ for any lattice point $P$ on the interior of $C_{k+1}P'_{k+1}$ we see that $h_{i,j}$ must be zero for any $(-i,-j)$ strictly between $D_{k+1}P_{k+1}$, indeed for such a point the quasi-greedy element $\ULX[i,j]$ contains the point $(-i,-j+ci)$ from the interior of $C_{k+1}P'_{k+1}$ in its support.  

The claim (i') now follows for $k+1$.  Indeed, the argument above implies $d^{(k+1)}_P=g_P=0$ for any lattice point $P$ in the interior of the line segment $C_{k+1}P'_{k+1}$.  But then Lemma \ref{lem:pq':q}(ii) gives
\begin{align*}
d^{(k+1)}_{P'_{k+1}}&=\ulc(p,q+(k+1)a_1-bp)_{(k+1)a_1,(k+1)a_2}=(\uc-\lc)(p,q)_{(k+1)a_1,(k+1)a_2}/2;\\
d^{(k+1)}_{P''_{k+1}}&=\ulc(p+(k+1)a_2-cq,q)_{(k+1)a_1,(k+1)a_2}=(\lc-\uc)(p,q)_{(k+1)a_1,(k+1)a_2}/2.
\end{align*}
We also see that only $\ULX[(k+1)a_1,(k+1)a_2]$ and $\ULX[(k+1)a_1-bp,(k+1)a_2-cq]$ (recall that $P_{k+1}=(-(k+1)a_1+bp,-(k+1)a_2+cq)$\; ) contribute to the coefficients $g_{P'_{k+1}}$ and $g_{P''_{k+1}}$, i.e. 
\[g_{P'_{k+1}}=d^{(k+1)}_{P'_{k+1}}+h_{-P_{k+1}}\quad\text{ and }\quad g_{P''_{k+1}}=d^{(k+1)}_{P''_{k+1}}+h_{-P_{k+1}}.\]
Adding these expressions and recalling \eqref{eq:2d P'} and \eqref{eq:2d P''} gives $0=0+2h_{-P_{k+1}}$, i.e. $h_{-P_{k+1}}=0$.  From this we get (ii') for $k+1$, i.e.
\[d^{(k+1)}_{P'_{k+1}}=g_{P'_{k+1}}=(k+1)(\uc-\lc)(p,q)_{a_1,a_2}/2=(k+1)d^{(1)}_{P'_1}.\]
This completes the inductive proof of (i') and (ii') and thus proves the lemma.
\end{proof}

\begin{lemma}\label{linear implies 0}\mbox{}

{\rm(i)} Both $\uc(p,q)$ and $\lc(p,q)$ are of the form
\begin{equation}\label{star}
\sum_{i,j=-N}^{N} c_{ij}(v)v^{ba_{2}i+ca_{1}j},
\end{equation}
where $N\in\mathbb{N}$ and $c_{ij}(v)\in \QQ(v)$ depend on $b,c,p,q$ but do not depend on $a_{1}, a_{2}$.
Therefore 
$$f_n:=f_n(v)=(\uc-\lc)(p,q)_{na_1,na_2}/2=\sum_{i=-M}^{M}d_{i}(v)v^{ni},$$
where $M\in\mathbb{N}$ and $d_{i}(v)\in\mathbb{Q}(v)$ does not depend on $n$. 

{\rm(ii)} For every positive integer $n$ we have $f_n(v)\equiv 0$.
\end{lemma}
\begin{proof}
(i) $\uc(p,q)$ (resp. $\lc(p,q)$) is a linear combination of products of quantum binomial coefficients of the form
$${a_{2}-cq'+k-1\brack k}_{v^{b}}\quad \textrm{ or } \quad {a_{1}-bp'+\ell-1\brack \ell}_{v^{c}}$$
where $p'\ge0,q'\ge0,k>0,\ell>0$ are integers. We may compute the first quantum binomial coefficient as
$${a_{2}-cq'+k-1\brack k}_{v^{b}}=\frac{\prod_{i=1}^{k}[a_{2}-cq'+k-i]_{v^{b}}}{\prod_{i=1}^{k}[i]_{v^{b}}}
=\frac{\prod_{i=1}^{k}(v^{b(a_{2}-cq'+k-i)}-v^{-b(a_{2}-cq'+k-i)})}{\prod_{i=1}^{k}(v^{bi}-v^{-bi})}
$$
which is of the form
$\sum_{i=-k}^{k} c_{i}(v)v^{ba_{2}i}$
where each $c_{i}(v)\in\QQ(v)$ does not depend on $a_{1}, a_{2}$. Similarly the second quantum binomial coefficient is of the form
$\sum_{j=-\ell}^{\ell} c_{j}(v)v^{ca_{1}j}$.
So  $\uc(p,q)$ (resp. $\lc(p,q)$) is of the given form \eqref{star}.  The second claim of (1) follows immediately.

(ii) Multiplying by a common denominator we can assume all $d_{i}(v)$ are polynomials in $v$.  Without loss of generality we assume $d_{M}(v)\neq0$. If $M>0$, then as $n$ grows to infinity, so does the degree of $v^{nM}$. But we have seen in Lemma~\ref{lem:linear}(ii) that $f_n=n\cdot f_1$ for all $n$ and, by the argument above, for $n$ sufficiently large the degree of $f_n(v)$ will be strictly larger than the degree of $f_1(v)$, a contradiction. Thus $M=0$ and $f_n(v)$ does not depend on $n$, in particular $f_1(v)=f_n(v)=nf_1(v)$ for all $n$, but this implies $f_1(v)\equiv 0$. 
\end{proof}

\begin{theorem}\label{theorem:greedy exists:q}
For any $(a_1,a_2)\in\ZZ^2$ we have 
\[\UX[a_1,a_2]=\LX[a_1,a_2]=\ULX[a_1,a_2]=X[a_1,a_2].\]
\end{theorem}
\begin{proof}
 If $(a_1,a_2)$ is not a positive imaginary root then the result was already proven in Corollary~\ref{cor:cluster monomials:q}.  Let $(a_{1},a_{2})$ be a positive imaginary root.  By Lemma~\ref{lem:linear}(ii) and Lemma~\ref{linear implies 0}(ii) we see that $\uc(p,q)_{a_1,a_2}=\lc(p,q)_{a_1,a_2}$ for every lattice point $(p,q)$ on the interior of the segment $OB$.  Now Lemma~\ref{lem:pq':q} gives $\ulc(p,q)=\uc(p,q)=\lc(p,q)=0$ for every $(p,q)$ on the boundary $\overline{\underline{R}}_{\text{greedy}}[a_1,a_2]\setminus R_{\text{greedy}}[a_1,a_2]$, i.e. each of $\UX[a_1,a_2]$, $\LX[a_1,a_2]$, $\ULX[a_1,a_2]$ satisfies the pointed support condition for $R_{\text{greedy}}[a_1,a_2]$.  The result now follows from Proposition~\ref{quasi-greedy_axiom} and Proposition~\ref{prop:equivalent}.
\end{proof}

We may now prove our main theorem.

\begin{proof}[Proof of Theorem \ref{main theorem1} and \ref{main theorem2}]

Theorem \ref{main theorem1} follows immediately from Theorem \ref{theorem:greedy exists:q} and Theorem~\ref{recursive-definition_u}.

The five parts in Theorem \ref{main theorem2} can be seen as follows.  Claim (a) follows from Proposition~\ref{mean invariant:q} and Theorem~\ref{theorem:greedy exists:q}.  For (b), we note that $X[a_1,a_2]$ is bar-invariant because, by definition, each coefficient $e(p,q)$, as well as each monomial $X^{(bp-a_1,cq-a_2)}$, is bar-invariant. The quantum greedy basis is independent of the choice of an initial cluster because $X[a_1,a_2]=\ULX[a_1,a_2]$ by Theorem \ref{theorem:greedy exists:q} and $\ULX[a_1,a_2]$ is mutation invariant by Proposition \ref{mean invariant:q}.  Claim ({c}) is the content of Corollary \ref{cor:cluster monomials:q}.  To see (d), assume that $X[a_1,a_2]$ is the sum of two universally positive elements in $\mathcal{A}_v(b,c)$. Since the specialization (by substituting $v=1$) of any universally positive element in $\mathcal{A}_v(b,c)$ is  universal positive in $\mathcal{A}(b,c)$,  the specialization $x[a_1,a_2]$ of $X[a_1,a_2]$ is also decomposable, in contradiction to Theorem \ref{main theorem-commutative} (d).  Claim (e) follows immediately by comparing the recurrence relations \eqref{eq:classical recurrence} and \eqref{eq:recurrence} defining the commutative greedy basis and quantum greedy basis respectively. 
\end{proof}

\section{Appendix: quantum binomial theorem}\label{ap:symmetric}
Here we recall some useful facts about quantum binomial coefficients.  Let $w$ be an indeterminate.
\begin{lemma} \label{le:quantum binomial theorem}
 Let $X$ and $Y$ be quasi-commuting variables with $YX=w^2XY$.  Then for any integer $n$ we have
 \begin{equation}
 (X+Y)^n=\sum\limits_{k\ge0}{n\brack k}_ww^{k(n-k)}X^kY^{n-k}.
\end{equation}
\end{lemma}
\begin{proof}
 We work by induction on $|n|$, the case $n=0$ being trivial.  For $n>0$ the following calculation accomplishes the induction step:
 \begin{align*}
  (X+Y)^n&=(X+Y)(X+Y)^{n-1}\\
  &=(X+Y)\sum\limits_{k\ge0}{n-1\brack k}_ww^{k(n-1-k)}X^kY^{n-1-k}\\
  &=\sum\limits_{k\ge0}\left({n-1\brack k-1}_ww^{(k-1)(n-k)}+{n\brack k}_ww^{k(n-1-k)+2k}\right)X^kY^{n-k}\\
  &=\sum\limits_{k\ge0}\left({n-1\brack k-1}_ww^{-n+k}+{n\brack k}_ww^k\right)w^{k(n-k)}X^kY^{n-k}\\
  &=\sum\limits_{k\ge0}{n\brack k}_ww^{k(n-k)}X^kY^{n-k}.
 \end{align*}
 But notice that reading the same calculation backwards proves the result for $n<0$.
\end{proof}

\begin{lemma}\label{cor:quantum convolution}
 For $m,n\in\ZZ$ and $k\ge0$ we have
 \begin{equation}
  {m+n\brack k}_w=\sum\limits_{r+s=k}w^{nr-ms}{m\brack r}_w{n\brack s}_w.
 \end{equation}
\end{lemma}
\begin{proof}
 Let $X$ and $Y$ be quasi-commuting variables with $YX=w^2XY$.  The left hand side of the desired equality is the coefficient of $w^{k(m+n-k)}X^kY^{m+n-k}$ in the product $(X+Y)^{m+n}$ while the right hand side is its coefficient in the product $(X+Y)^m(X+Y)^n$.
\end{proof}

\end{document}